 \documentclass[12pt, twoside, leqno]{article}

\usepackage{amsmath,amsthm}
\usepackage{amssymb}

\usepackage{enumitem}
 
\usepackage{graph icx}

\usepackage[T1]{fontenc}

\usepackage{hyperref} 
\newtheorem{theorem}{Theorem}[section]
\newtheorem{corollary}[theorem]{Corollary}
\newtheorem{lemma}[theorem]{Lemma}
\newtheorem{proposition}[theorem]{Proposition}

\theoremstyle{definition}
\newtheorem{definition}[theorem]{Definition}
\newtheorem{remark}[theorem]{Remark}

\numberwithin{equation}{section}

\numberwithin{equation}{section} 
\newcommand {\Q}{{\mathbb{Q}}}
\newcommand {\C}{{\mathbb{C}}}
\newcommand {\PP}{{\mathbb{P}}}
\newcommand {\R}{{\mathbb{R}}}
\newcommand {\Z}{{\mathbb{Z}}}  
\newcommand{\rmd}{\mathrm {d}}
\newcommand{\rme}{\mathrm {e}}
\newcommand{\rmB}{\mathrm {B}}
\newcommand{\Aut}{\mathrm {Aut}}
\newcommand{\Bin}{\mathrm {Bin}}
\newcommand{\Coef}{\mathrm {Coef}}
\newcommand{\Gl}{\mathrm {GL}}
\newcommand{\Id}{\mathrm {Id}} 
\renewcommand{\Im}{\mathrm {Im}}
\newcommand{\PGl}{\mathrm {PGL}}
\newcommand{\calB}{\mathcal {B}}

\newcommand{\calD}{\mathcal {D}}
\newcommand{\calE}{\mathcal {E}} 
\newcommand{\calF}{\mathcal {F}}
\newcommand{\calI}{\mathcal {I}}
\newcommand{\calL}{\mathcal {L}}
\newcommand{\calM}{\mathcal {M}}
\newcommand{\calN}{\mathcal {N}}
\newcommand{\calQ}{\mathcal {Q}}
\newcommand{\calR}{\mathcal {R}}
\newcommand{\calZ}{\mathcal {Z}}
\newcommand{\Bir}{\mathrm {Bir}}

\newtheorem{theorema}{Theorem}[section]

\def\atop#1#2{
\genfrac{}{}{0pt} {} 
{#1} 
{#2}}

\def\virgule{,}

\begin{document}


\baselineskip=17pt


\title{Number of integers represented by families of binary forms}

\author{\'Etienne Fouvry\\
Univ. Paris--Sud, CNRS, Universit\' e Paris--Saclay, CNRS,    \\
Laboratoire  de  Math\' ematiques  d'Orsay,     \\
91405   Orsay,   France \\
E-mail: Etienne.Fouvry@universite-paris-saclay.fr
\and 
Michel Waldschmidt\\
Sorbonne Universit\' e and Universit\' e de Paris, \\
CNRS, IMJ--PRG, \\
75005 Paris, France\\
E-mail: michel.waldschmidt@imj-prg.fr}

\date{\today}

\maketitle

\hfill{\em Nous d\'edions ce travail \`a la m\'emoire d'Andr\'e Schinzel }

\hfill{\em avec notre profond respect et notre  affectueuse admiration}

\renewcommand{\thefootnote}{}

\footnote{2020 \emph{Mathematics Subject Classification}: Primary 11E76; Secondary 11D45 11D85.}

\footnote{\emph{Key words and phrases}: Binary forms, Representation of integers by binary forms, Families of Diophantine equations.}

\renewcommand{\thefootnote}{\arabic{footnote}}
\setcounter{footnote}{0}

\begin{abstract}  
We extend our previous results on the number of integers which are values of some cyclotomic form of degree larger than a given value (see \cite{FW1}), to more general families
of binary forms with integer coefficients. Our main ingredient is an asymptotic upper bound for the cardinality of the set of values which are common to two non--isomorphic
binary forms of degree greater than $3$. We apply our results to  some typical examples of families of binary forms.
  \end{abstract}
  

\section{Introduction} 
Let $d\geq 3$ be an integer. We denote by ${\Bin} (d, \Z)$ 
the set of binary forms $F= F(X,Y)$ with integer coefficients,  of degree $d$ and with  discriminant different from
zero.   For 
\begin{equation}\label{defgamma}
\gamma =\begin{pmatrix} a_1&a_2\\a_3&a_4
\end{pmatrix} \in {\Gl}(2, \Q),
\end{equation}
 and $F\in {\Bin}(d,\Z)$, $F\circ \gamma$ is the binary form with rational coefficients, defined by
$$
\left(  F\circ \gamma \right) (X_1,X_2)=
F(a_1X_1+a_2X_2, a_3X_1 +a_4X_2).
$$
Two elements $F_1$ and $F_2$ in ${\Bin}(d,\Z)$ are said to be {\em isomorphic} is there is a $\gamma \in {\Gl}(2,\Q)$ such that
$$
F_1 \circ \gamma =F_2.
$$
To  estimate the number of values simultaneously taken by  the binary forms  $F_1$ and $F_2$, we introduce the counting function, for $N$ an integer $\ge 1$, 
\begin{equation}\label{defNF1=F2}
\begin{aligned}
{\calN} (F_1\virgule F_2; N) &:=\sharp\left( F_1 (\Z^2) \cap F_2 (\Z^2 ) \cap [-N, +N] \right)\\
& = \sharp \bigl\{ m: \vert m \vert \leq N, \text{ there exists }(x_1,x_2,x_3,x_4) \in \Z^4
\\
& \hskip 2cm \text{ such that } m=F_1(x_1,x_2) =F_2 (x_3,x_4)\bigr\}.  
\end{aligned}
\end{equation}
Our first result gives an upper bound for this function when the two forms are not isomorphic.

 \begin{theorem}\label{keyresult} For every $d \geq 3$, there is a constant   $\vartheta_d <2/d$ such that, for every $\varepsilon >0$, for every pair
$(F_1,F_2)$ of no--isomorphic forms of ${\Bin}(d, \Z)$, for  $N \to\infty$, one has the bound
\begin{equation}\label{140}
{\calN}\left( F_1\virgule F_2;N\right) =O_{F_1,F_2,\varepsilon} \left( N^{\vartheta_d +\varepsilon}\right).
\end{equation}
\end{theorem}
This theorem calls for the following comments:
\begin{remark} The point in this theorem is  that the constant $\vartheta_d$, defined in \eqref{valueofthetad}, satisfies the inequality $\vartheta_d < 2/d$ (see the inequalities 
\eqref{eta<2/d} below). In fact, it is known that for any $F \in {\Bin}(d,\Z)$, there exists $C_F >0$, such that, for $N$ tending to infinity, one has the equality
\begin{equation*} 
{\calN}(F\virgule F;N) =\left( C_F +o_F(1) \right) N^{2/d},
\end{equation*}
(see Theorem \ref{Stewart} in \S \ref{para1.1}, due to Stewart and Xiao \cite[Theorem 1.1]{St-Xi}).
\end{remark}
\begin{remark} The explicit value of $\vartheta_d$ given in \eqref{valueofthetad} leads to the inequality $\vartheta_d >1/d$ for all  $d\geq 3$ (see \eqref{eta<2/d}). It also shows that $\vartheta_d$ is asymptotic
to $1/d$ as $d\to\infty$. This value is asymptotically optimal as shown by the two forms
$$
F_1 (X,Y) =X^d+Y^d \text{ and } F_2 (X,Y) = X^d +2 Y^d.
$$
These two forms are not isomorphic. From the equalities $F_1 (n,0)=F_2 (n,0) =n^d$, we deduce the lower bound
$$
{\calN} (F_1\virgule F_2;N)\geq N^{1/d}\ (N\geq 1).
$$
\end{remark}
\begin{remark} 
 According to \cite[Corollaire 3.3]{FW1}, if the two forms $F_1$, $F_2$ are positive definite and at least one of them is not divisible by a linear form with rational coefficients, then the exponent $\vartheta_d$ in the conclusion of Theorem \ref{keyresult}  can be replaced by $\eta_d$ with $\eta_d<\vartheta_d$ (see the definition of $\eta_d$  and $\vartheta_d$ in \S\ref{SS:Beginning}).
\end{remark}
\begin{remark}  \label{Remark:NoRealRoot}
We will show in \S\ref{2.4} that the exponent $\vartheta_d$ in the conclusion of Theorem \ref{keyresult}  can be replaced by  the coefficient $\eta_d$ quoted in the previous remark when the binary form $F_1(X,Y)F_2(X,Y)$ has no real root. 
\end{remark}
\begin{remark} Theorem \ref{keyresult} is no more valid for $d=2$. This is well known: see for instance \cite[Prop. 6.1, eq. (6.3)]{FLW}, where, choosing 
$F_1 (X,Y) = X^2+Y^2$ and $F_2 (X,Y)= X^2 +XY +Y^2$, one has, for $B$ tending to infinity, the asymptotic formula
$$
{\calN} (F_1\virgule F_2;B) = \left( \beta_0 +o (1) \right) B\big(\log B)^{-3/4},
$$
for some constant $\beta_0 >0$.
\end{remark}
\begin{remark} Theorem \ref{keyresult}  is immediately generalized to binary forms with rational coefficients, it suffices to multiply by a common denominator.  
\end{remark}
\begin{remark} The following proposition shows that if $F_1$ and $F_2$ are isomorphic, the equality \eqref{140} never holds.

\begin{proposition}\label{F1etF2isom}
 Let $d\geq 3$ and let $F_1$ and $F_2$ be two isomorphic binary forms in ${\Bin}(d, \Z)$. Then there is  a positive constant $C_{F_1,F_2}$, such that, for $N$ tending to infinity, we have
the inequality
$$
{\calN}( F_1\virgule F_2; N) \geq \left( C_{F_1, F_2} -o_{F_1,F_2}(1) \right) N^{2/d}.
$$
\end{proposition}

 \begin{proof} Let $\gamma$ as in \eqref{defgamma} such that $F_1 =F_2 \circ \gamma$. Let $D\geq 1
$ be an integer such that $(Da_1, Da_2, Da_3, Da_4)$ belongs to $\Z^4$. By homogeneity, we deduce that the two forms
\[
G_1 (X_1,X_2) :=F_1 (DX_1, DX_2) 
\]
and
\[
G_2 (X_1,X_2) := F_2 \left( Da_1X_1+Da_2X_2, Da_3X_1+Da_4 X_2\right)
\]
are equal. So we have the equality of their images
$$
G_1 (\Z^2) = G_2 (\Z^2).
$$ We also have the obvious inclusions
$$
G_1 (\Z^2) \subset F_1 (\Z^2) \text{ and } G_2 (\Z^2) \subset F_2 (\Z^2),
$$
which lead to the inclusion
\begin{equation}\label{*}
G_1 (\Z^2) \subset F_1 (\Z^2) \cap F_2 (\Z^2).
\end{equation}
A new application of  the result of Stewart and Xiao (see Theorem \ref{Stewart} below) gives, for some
constant  $C_{G_1} >0$, the equality
\begin{equation}\label{**}
{\calN}( G_1\virgule G_1; N) =\left(C_{G_1}+o_{G_1}(1) \right) N^{2/d},
\end{equation}
as $N$ tends to infinity.  Gathering \eqref{*}  and  \eqref{**} we obtain the inequality claimed in Proposition \ref{F1etF2isom}.
\end{proof}
\end{remark}

Theorem \ref{keyresult} is an important tool for our generalisation of our previous study in \cite{FW1}, 
where we produced an asymptotic formula for the number of values $m$, with $\vert m \vert \leq B$  taken by some cyclotomic form $\Phi_n$ of degree  $\varphi (n) $ greater than a fixed $ d \geq 3$. 
Recall that  $\varphi$ is the Euler function and that to the $n$--th cyclotomic polynomial $\phi_n (X)$, of degree $\varphi (n)$,  is attached the {\em cyclotomic form } $\Phi_n (X,Y):= Y^{\varphi (n)} \cdot \phi_n (X/Y)$.

Our purpose is to study the following  general problem:
\vskip .3cm
{\em Let $\calF$ be an infinite subset of $\displaystyle{\underset {d\geq 3}{\bigcup}{\Bin}(d, \Z)}$, satisfying natural properties. Let $
A$ be a fixed non negative integer. As $B$ tends to infinity, estimate the counting function
\begin{equation}\label{defRgeqd}
\begin{aligned}
{\calR}_{\geq d}  \left (\calF, B, A\right) := \sharp\, \bigl\{ m: 0\leq \vert m \vert \leq B, \, \text{ there is }   \, F\in \calF \text{ with } \deg F \geq d
\\
\text{ and } (x,y)\in \Z^2 \text{ with } \max \{\vert x \vert, \vert y \vert\} \geq A, \text{ such that } F(x,y) =m
\bigr\}. 
\end{aligned}
\end{equation} }
\vskip .3cm
The introduction of the parameter $A$ may be seen as artificial. It  is enacted to prevent from the following  phenomenon encountered for instance in the case of the family of cyclotomic forms  $\Phi_n$, where, for every prime $p$,
we have 
\begin{equation}\label{valueofPhi}
\Phi_p (1,1) =p\end{equation} (recall that $ \Phi_p (X,Y) = (X^p-Y^p)\big/ (X-Y)$). We wish to avoid counting these values, since the set of primes, by its cardinality, completely hides the set of  other values $\Phi_n (x,y)$
when $\max \{\vert x\vert, \vert y \vert \}\geq 2$ and $\varphi (n) \geq d$.

Let $\calF$ be a set of binary forms. We denote by $\calF_d$ the subset of forms of $\calF$ of degree $d$. We will study the set of values taken by forms belonging to some 
  {\em $(A,A_1, d_0, d_1,  \kappa)$--regular  } families $\calF$, that we define as follows.
\begin{definition}\label{agreable} Let $A$, $A_1$, $d_0$, $d_1$ be integers and let $\kappa$ be a real number such that
\begin{equation}\label{condAA1d0kappa}
A\geq 1, \, A_1\geq 1, \,d_1\geq  d_0 \geq 0, \, 0< \kappa <A.
\end{equation}
 Let $\calF$ be a set of binary forms. We say that $\calF$ is {\em$(A,A_1, d_0, d_1, \kappa)$--regular} if it satisfies the following conditions:
 \begin{enumerate}[label=\upshape(\roman*), leftmargin=*, widest=iii]
\item
\label{it:1}
The set $\calF$ is infinite,
\item 
\label{it:2}
We have the inclusion
$$
\calF \subset  \bigcup_{d\geq 3} {\Bin} (d,\Z),
$$
\item 
 \label{it:3}
For all $d\geq 3$, one has  the inequality $\sharp \calF_d  \leq d^{A_1} $,  
\item 
\label{it:4}
Two forms of $\calF$ are isomorphic if and only if they are equal,
\item 
\label{it:5}
For any  $d\ge \max\{d_1,d_0+1\}$, the following holds
$$
\left. 
\begin{array}{ll}
F\in \calF_d ,   \\
(x,y) \in \Z^2 \text{ and }  F(x,y) \ne 0, \\
\max\{\vert x\vert, \vert y \vert\} \geq A,
\end{array}
\right\}
 \Rightarrow \max\{\vert x \vert, \vert y \vert \}\leq \kappa \left\vert F(x,y)\right\vert^{\frac 1{d-d_0}}. 
$$ 
\end{enumerate}
\end{definition}
The upper bound in the right hand side of (v) is trivial for $\max \{|x|,|y|\}\le \kappa$, this is why we request $A>\kappa$. 
 
 The family of cyclotomic forms
 $$
 \boldsymbol \Phi := \{ \Phi_n:  \varphi (n)\geq 4, \; n \not\equiv 2 \pmod 4\}
 $$
 satisfies the assumptions {\rm (i)},  {\rm (ii)},  {\rm (iii)},  {\rm (iv)}, but 
 is not $(1,A_1, d_0,d_1,  \kappa)$--regular for any value of  $A_1$, $d_0$, $d_1$  and $\kappa$, since \eqref{valueofPhi} shows that {\rm (v)} is not satisfied.  However $\boldsymbol \Phi$ is $(2,2, 0, 4, 2/\sqrt 3)$--regular, this is a consequence of  \cite[Th\'eor\`eme 4.10]{FW1} and of the classical inequality $n/(\log \log n) <\varphi (n) <n$.

 \subsection{Some facts on a single form}\label{para1.1}
 Before stating our main result concerning $\calR_{\geq d} (\calF,B,A) $ defined in \eqref{defRgeqd}, we recall some fundamental objects attached to a binary form $F \in {\Bin} (d,\Z)$
 when $d \geq 3$:
 \begin{itemize}
\item The  {\em fundamental domain} of  $F$ is 
$$
\calD (F) := \left\{ (x,y) \in \R^2 : \vert F(x,y) \vert \leq 1 \right\},
$$
\item The {\em area of the fundamental domain  of $F$} is  the real number
\begin{equation}\label{defAF}
A_F := \iint_{\calD (F)} {\rmd} x\, {\rmd} y.
\end{equation}
We always have $0 < A_F < \infty$.
\item The   {\em group of automorphisms } of $F$ is 
$$
\begin{aligned}
{\Aut}& (F;\Q):=
\\
& \left\{ \begin{pmatrix} a_1&a_2\\a_3& a_4\end{pmatrix}\in {\Gl}(2, \Q): F(X,Y)= F(a_1X+a_2Y, a_3 X+a_4 Y) \right\}.
\end{aligned}
$$
\end{itemize}
This is a finite subgroup of ${\Gl}(2, \Q)$. 
We now recall the important result of Stewart and  Xiao, that we already mentioned above \cite[ Theorems 1.1 and 1.2]{St-Xi}:
\begin{theorema}\label{Stewart} 
\emph{For every  $d\geq 3$, there is a constant $\kappa_d < 2/d$ such that, for all     $F\in {\Bin}(d, \Z)$ and for all $\varepsilon>0$,  the following equality 
$$
\calN (F\virgule F;B)= A_F\cdot  W_F \cdot B^{2/d} +O_{F,\varepsilon} \left(B^{\kappa_d +\varepsilon}\right),
$$ 
holds uniformly for  $B \to\infty$, where $W_F=W \left( {\Aut} (F; \Q)\right)$ depends only on the group $ {\Aut} (F; \Q)$.}
\end{theorema}

For $G$ a finite subgroup of ${\Gl}(2, \Q)$  which is the group of automorphisms of an element of $ {\Bin}(d, \Z)$, the  constant   $W(G)$  is a  rational number  which is defined
  in  \cite[Theorem 1.2]{St-Xi}. This definition is subtle since it depends on the denominators of the entries of the matrices belonging to   $G$. However
 for the   families $\calF$ that we will meet in 
  this paper, we will only need the equalities
  \begin{equation}\label{W=1W=1}
W\left(\{ \Id\}\right) =1,\,   W\left( \{\Id, \, -{\Id}\} \right)=1/2  \text{ and } 
 W\left(
 \left\{
  \begin{pmatrix}
   \pm 1 & 0\\ 0& \pm 1 
 \end{pmatrix} 
 \right\}
  \right)
  =1/4.
\end{equation}
Finally, the exponent $\kappa_d$ in Theorem \ref{Stewart} is defined by
 \begin{equation}\label{valueofkappad}
 \kappa_d=
 \begin{cases}
 \displaystyle{ \frac {12}{19} }&\text{ if } d=3,\\
\displaystyle{ \frac 3{(d-2)\sqrt d +3} }&\text{ if } 4\leq d \leq 8,\\
\displaystyle{\frac 1{d-1}} &\text{ if } d\geq 9. 
 \end{cases}
 \end{equation} Actually the value of  this exponent is  improved when  $F(X,Y)$ does not have a linear factor over   $\R [X,Y]$, see \cite[ formula (1.11)]{St-Xi}.
\subsection{An asymptotic formula for ${\calR}_{\geq d} (\calF,B,A)$}
Our central result is the following. The exponent $\vartheta_d$  is defined   in \eqref{valueofthetad}. 

\begin{theorem}\label{Thm-on-R}
Let $(A,A_1, d_0, d_1, \kappa)$ satisfying the conditions  \eqref{condAA1d0kappa}.   Let $\calF$ be a $(A,A_1,d_0, d_1, 
\kappa)$--regular family of binary forms. Then for every $d\geq  \max\{3, d_1\}$  and every positive $\varepsilon$, 
  one has the equality 
$$
\calR_{\geq d} (\calF, B, A) = \left( \sum_{F\in \calF_d} A_F W_F \right)   \cdot B^{2/d} + O_{\calF,A,d, \varepsilon}\bigl( B^{\vartheta_d+\varepsilon} \bigr)  +   O_{\calF,A,d} \bigl( B^{2/d^\dag}\bigr), 
$$
uniformly for $B  \to\infty$.
The  integer $d^\dag$ is   defined by
$$
d^\dag :=\inf\{ d' : d' >d \text{ such that } \calF_{d'} \ne \emptyset\}.
$$
\end{theorem}
Recall that   $\calF_d$ is not empty for infinitely many values of $d$ since  the set $\calF$ is infinite. 

The assumption (v) in the definition  \ref{agreable} of a regular family cannot be omitted, even in the case of totally imaginary forms (homogeneous versions of polynomials without real roots), as show by the sequence of positive definite forms $(X-Y)^2(X-2Y)^2\cdots (X-dY)^2+dY^{2d}$, the value of which at the points $(x,y)=(n,1)$, $1\le n\le d$, is $d$.

The following is a direct application of \eqref{W=1W=1}:
\begin{corollary}\label{11/21/4}
Suppose that $\calF$ satisfies the hypothesis of Theorem \ref{Thm-on-R} and that,  for every $d\geq 3$, $\calF_d$   satisfies one of the three following conditions:
\begin{itemize}
\item[{\em C1}:] for all $F \in \calF_d$, we have ${\Aut} (F, \Q)= \{ {\Id}\}$,
\item[{\em C2}:] for all $F \in \calF_d$, we have ${\Aut} (F, \Q)= \{ \pm {\Id}\}$ (cyclic group of order $2$),
\item[{\em C3}:] for all $F \in \calF_d$, we have ${\Aut} (F, \Q)= \left\{ \begin{pmatrix} \pm1 & 0\\ 0 & \pm1\end{pmatrix}\right\}$ (Klein group of order $4$).
\end{itemize}
Then we have the equality
\begin{equation}\label{restricted}
\calR_{\geq d} (\calF, B, A) = C_d\cdot\left( \sum_{F\in \calF_d} A_F  \right)   \cdot B^{2/d} + O_{\calF,A,d, \varepsilon}\bigl( B^{\vartheta_d+\varepsilon} \bigr)  +   O_{\calF,A,d} \bigl( B^{2/d^\dag}\bigr) ,
\end{equation}
where the coefficient  $C_d $ has respectively the values $1$, $1/2$ or $1/4$ according to the condition {\em C1},  {\em C2},  {\em C3} respectively satisfied by  $\calF_d$.
 \end{corollary}

 \subsection{Some applications}  We now give a list of regular families $\calF$ in order to illustrate our results. 
 
 The first example of course is given by the sequence of cyclotomic binary forms \cite{FLW}. We do not repeat it. 
 
 Our second example is given by a family of binomials $ax^d+by^d$ where $d$ is even while $a$, $b$ have the same sign: these restrictions allow us to check easily the assumption {\rm (v)} in the definition \ref{agreable} of a regular family. Since the proof is easy, we give it right away. 
 
 The three other examples below will require more work; for them we restrict ourselves to families $\calF$ satisfying the conditions of Corollary \ref{11/21/4} in order to
 apply \eqref{restricted}. 
 
 There are a lot of variations on these constructions.

 \subsubsection{Binomial forms} \label{SS:BinomialForms}
 For each even integer $d\ge 4$, let $\calE_d$ be a finite subset of $\Z_{>0}\times\Z_{>0}$. Assume $\calE_d$ is not empty for infinitely many $d$ and has at most $d^{A_1}$ elements  for some $A_1>0$ and all $d$. Let $\calB_d$ denote the family of binary forms $aX^d+bY^d$ with $(a,b)\in\calE_d$ and let $\calB=\cup_{d\ge 4}\calB_d$.
 We assume that for $(a,b)\not=(a',b')$ in $\calE_d$,   one at least of $a/a'$, $b/b'$ is not a $d$-th power of a rational number, and also one at least of $a/b'$, $b/a'$ is not a $d$-th power of a rational number.
 
  \begin{theorem}\label{CaseofB} The family $\calB$ is $(2,A_1,0,4, 1)$--regular.
 
  Further,  for  every $d\geq 4$ and for every $\varepsilon >0$ we have the equality 
 \begin{equation}\label{caseofB}
 \calR_{\geq d}   (\calB, B, 2) =   \left( \sum_{F\in \calB_d} A_F W_F\right) B^{2/d} + 
 O_{\calB,d, \varepsilon}\left( B^{\max\{ \vartheta_d +\varepsilon, 2/d^{\dag} \}}\right),
 \end{equation}
  uniformly for $B \to\infty$.  The  integer $d^\dag$ is   defined by
$$
d^\dag :=\inf\{ d' : d' >d \text{ such that } \calB_{d'} \ne \emptyset\}.
$$
 \end{theorem}
 
 We will check the hypothesis  {\rm (iv)} of the definition \ref{agreable} of a regular family by means of the following auxiliary result. 

\begin{lemma}\label{Lemma:isomorphismB}
Let $d\ge 4$ be even and let $a,b,a',b'$ be positive integers. Then the two binary forms $aX^d+bY^d$ and $a'X^d+b'Y^d$  are isomorphic if and only if either  $a/a'$, $b/b'$ are both $d$-th power of  rational numbers, or  $a/b'$, $b/a'$ are both $d$-th power of  rational numbers.
\end{lemma}

 \begin{proof}
 If $a/a'=u^d$ and  $b/b'=v^d$, then the two forms $aX^d+bY^d=a'(uX)^d+b'(vY)^d$ and $a'X^d+b'Y^d$  are isomorphic. Also 
 if $a/b'=u^d$ and  $b/a'=v^d$, then the two forms $aX^d+bY^d=a'(vY)^d+b'(uX)^d$ and $a'X^d+b'Y^d$  are isomorphic. It remains to prove the converse.
 
 Assume that the two binary forms $aX^d+bY^d$ and $a'X^d+b'Y^d$  are isomorphic.
Let $\gamma=\begin{pmatrix}a_1&a_2\\a_3&a_4\end{pmatrix}\in\Gl(2, \Q)$ satisfy
$$
a(a_1X+a_2Y)^d+b(a_3X+a_4Y)^d=a'X^d+b'Y^d.
$$
We have
$$
a a_1^d+ba_3^d=a', \quad aa_2^d+ba_4^d=b'
$$
and, for  $i=1,\dots, d-1$, 
$$
aa_1^ia_2^{d-i}+ba_3^ia_4^{d-i}=0.
$$
$\bullet$ Assume $a_2=0$. From
$$
a(a_1X)^d+b(a_3X+a_4Y)^d=a'X^d+b'Y^d
$$
we deduce $ba_4^d=b'$, $a_4\not=0$,  hence $a_3=0$, and therefore $aa_1^d=a'$.
\\
$\bullet$ Assume $a_1=0$. From
$$
a(a_2Y)^d+b(a_3X+a_4Y)^d=a'X^d+b'Y^d
$$ 
we deduce $ba_3^d=a'$, $a_3\not=0$, hence  $a_4=0$, and therefore $aa_2^d=b'$.
\\
$\bullet$ Finally let us check that the case $a_1a_2\not=0$ is not possible. Write
$$ 
aa_1a_2^{d-1}+ba_3a_4^{d-1}=0, \quad
aa_1^2a_2^{d-2}+ba_3^2a_4^{d-2}=0.
$$
We deduce $a_3a_4\not=0$, 
$$
\frac {a_1}{a_3}=-\frac b a\left(\frac {a_4}{a_2}\right)^{d-1},\quad
\left(\frac {a_1}{a_3}\right)^2=-\frac b a \left(\frac {a_4}{a_2}\right)^{d-2},
$$
hence
$$
\left(\frac {a_4}{a_2}\right)^d=-\frac a b
$$
which is impossible for  $a$, $b$ positive and $d$ even. 

 \end{proof}
 
 \begin{proof}[Proof of Theorem \ref{CaseofB}]
The conditions {\rm (i)}, {\rm (ii)} and {\rm (iii)} in the definition \ref{agreable} of a regular family are satisfied by hypothesis. 

For $(a,b)\not=(a',b')$ in $\calE_d$,  the two binary forms  $aX^d+bY^d$ and $a'X^d+b'Y^d$ are not isomorphic, as shown by Lemma \ref{Lemma:isomorphismB}. 
Finally, for $(a,b)\in\calE_d$ and $(x,y)\in\Z^2$, we have
$$
ax^d+by^d\ge \max\{|x|,|y|\}^d.
$$
This completes the proof of condition {\rm (v)} in the definition \ref{agreable} of a regular family. 

The second assertion of  Theorem \ref{CaseofB} then follows from Theorem  \ref{Thm-on-R}.
 \end{proof}
 
 Our assumptions do not allow any upper bound for $ \calR_{\geq d}   (\calB, B, 1)$ better than $B$: the set of all $a$, $b$ and $a+b$ for $(a,b)$ in $\cup_{d'\ge d} E_{d'}$ may contain all positive integers. 
 
 Explicit values for $W_F$ and $A_F$  for $F\in \calF_d$  are given in Corollary  1.3 of  \cite{St-Xi}.
 The values of $W_F$ and $A_F$ are computed without the two assumptions $a,b$ of the same sign and $d$ even, but none of these two hypotheses can be omitted for our theorem, as shown by the two sequences
$X^d-(d^d-d)Y^d$ ($d$ even)
and 
$X^d+(d^d-d)Y^d$ ($d$ odd).

 \subsubsection{Products of positive quadratic forms} \label{SS:ProductPositiveQuadraticForms}
 Let $(\mu_n)_{n\ge 1}$ be an increasing sequence of positive squarefree integers; assume that there exists $\lambda>0$ such that 
 \begin{equation}\label{mun<}
 \mu_n\le \lambda \, n \text{ for all } n\ge 1.
 \end{equation}
  If we choose $\mu_n=q_n$ where  $(q_n)_{n\ge 1}$ is the full sequence 
 $$
 1,\, 2, \, 3,\, 5,\, 6, \, 7,\, 10,\, 11, \, 13,\, 14,\, 15,...
 $$
  of positive squarefree integers, written in ascending order, then, as it is well known   (see \cite[Theorem 333]{HW} and  \url{https://oeis.org/A005117}), we have 
$$
\sharp \{ q_n \leq x\}= \sum_{n\leq x} \mu(n)^2  = \frac{6}{\pi^2} x + O (\sqrt x),
$$
which implies that
$$
q_n\sim \frac{\pi^2}6 n \; (n \rightarrow \infty).
$$
Since $\mu_{230}\ge q_{230} =381$, we have $\lambda\ge \frac {381}{230}\cdotp$ As a matter of fact we have 
\begin{equation}\label{sup-q_n/n}
\sup_{n\ge 1} \frac {q_n} n=
\frac {381}{230}\cdotp
\end{equation}
Hence, in the special case $\mu_n=q_n$ ($n\ge 1$), $\lambda = 381/230$ is an admissible value.  

 For $d\geq 2$ and $1\leq \nu \leq d+1$,
 we denote by
 $
 Q^+_{d, \nu}$ 
 the binary form of degree $2d$ defined by the formula
 \begin{equation}\label{defQ+}
 Q^+_{d, \nu} (X,Y) := \prod_{\atop{1\leq n \leq d+1}{ n\ne \nu}} \left( X^2 + {\mu}_n Y^2\right). 
 \end{equation}
 The associated family is
 $$
 \calQ^+:= \left\{ Q^+_{d, \nu} : d\geq 2,\, 1 \leq \nu \leq d+1\right\}
 $$
 with $ \calQ^+_d=\emptyset$ for $d$ odd and  $ \calQ^+_{2d}= \left\{ Q^+_{d, \nu} :  1 \leq \nu \leq d+1\right\}$ for $d\ge 2$.
 With $\lambda$ defined in \eqref{mun<}, we have
 
 \begin{theorem}\label{CaseofQ+} The family $\calQ^+$ is $(2,1,0,4, 1)$--regular.
 
  Furthermore, for every $d~ \geq 2$, $\calQ_{2d}^+$ satisfies the condition {\em C3}  of  Corollary~\ref{11/21/4}. 
  
  Finally,  for  every $d\geq 2$ and for every $\varepsilon >0$ we have the equality 
 \begin{equation}\label{397}
 \calR_{\geq 2d}   (\calQ^+, B, 0) = \frac 14 \left( \sum_{F\in \calQ_{2d}^+} A_F\right) B^{1/d} + 
 O_{\lambda,d, \varepsilon}\left( B^{\max\{ \vartheta_{2d} +\varepsilon, 1/(d+1)\}}\right),
 \end{equation}
  uniformly for $B \to\infty$,
 and the inequalities 
\begin{equation}\label{sqrt6}
 \frac \pi {\sqrt \lambda}  \cdot  \sqrt d < \left( \sum_{F\in \calQ_{2d}^+} A_F\right) <  \pi\sqrt{\rme}  ( \sqrt d+1).
\end{equation}  
 \end{theorem}
See \eqref{eta'<} for a simplification of the exponent in the error term of \eqref{397}.

\begin{remark}  \label{Remark:JB}
({\em Thanks to Jean-Baptiste Fouvry}). 
Consider the two quartic forms 
$$
Q_{2,3}^+(X,Y)=(X^2+Y^2)(X^2+2Y^2)
\quad
\text{and}\quad
Q_{2,1}^+(U,V)=(U^2+2V^2)(U^2+3V^2).
$$
One checks
$$
Q_{2,3}^+(X,Y)-Q_{2,1}^+(U,Y)=(-U^2+X^2-Y^2)(U^2+X^2+4Y^2).
$$
The Pythagorean triples  $(y,u,x)$, namely the solutions of the equation $y^2+u^2=x^2$, produce solutions $(m,x,y,u)$ to the equations 
$$
m=Q_{2,3}^+(x,y)=Q_{2,1}^+(u,y).
$$
It follows that the exponent $\vartheta_4 = 0.448$ in Theorem  \ref{keyresult}  cannot be replaced with an exponent $<0.25$. 
\end{remark}

 \subsubsection {Products of indefinite quadratic forms} 
 With the above notations, including the definition of $\lambda$ in \eqref{mun<}, we assume $\mu_1\ge 2$ and we  consider, for $d\ge 2$ and $1\le \nu\le d+1$,  the binary form of degree $2d$
 $$
 Q^-_{d, \nu} (X,Y) := \prod_{\atop{1\leq n \leq d+1}{ n\ne \nu}} \left( X^2 - {\mu}_n Y^2\right). 
 $$
 The associated family is
 $$
\calQ^-:= \left\{ Q^-_{d, \nu} : d\geq 2,\, 1 \leq \nu \leq d+1\right\}
 $$
 with $ \calQ^-_d=\emptyset$ for $d$ odd and  $ \calQ^-_{2d}= \left\{ Q^-_{d, \nu} :  1 \leq \nu \leq d+1\right\}$ for $d\ge 2$.

From \eqref{sup-q_n/n} one deduces
$$  
\sup_{n\ge 1} \frac {q_{n+1}} n=2,
$$
hence $\lambda\ge 2$. In the special case $\mu_n=q_{n+1}$ ($n\ge 1$), an admissible value for $\lambda$ is $\lambda=2$.

 \begin{theorem}\label{CaseofQ-} For $A>2\rme\lambda$, the family $\calQ^-$ is $(A,1,2,2,2\rme\lambda)$--regular 
 and satisfies the condition {\em C3} of Corollary \ref{11/21/4}. Furthermore, for $d\ge 2$, we have   
 $$
 \calR_{\geq 2d} (\calQ^-, B, 0 )=\frac 14 \left( \sum_{F\in \calQ_{2d}^-} A_F \right) B^{1/d} + 
 O_{\lambda,A,d, \varepsilon}\left( B^{\max\{ \vartheta_{2d} +\varepsilon, 1/(d+1)\}}\right),
 $$ 
  uniformly for $B  \to\infty$. Further,  we have 
\begin{equation}\label{averageAFQ-}
 \frac \pi {\sqrt \lambda}  \cdot  \sqrt d  \leq  \sum_{F\in \calQ_{2d}^-} A_F \le
 22\lambda \sqrt d
\end{equation} 
where the lower bound is valid for all $d\ge 2$ and the upper bound  for $d$ sufficiently large.
  \end{theorem}

 \subsubsection{Products of linear factors}\label{1.3.3} 
 We reserve the letter $p$ to prime numbers and we consider for $5\leq d \leq p$, the binary form $L_{d,p}\in {\Bin} (d,\Z)$
 defined by
 $$
 L_{d,p} (X,Y):= \left( X-pY\right) \cdot \prod_{0\leq n \leq d-2} \left( X-nY \right).
 $$
The associated family is 
 $$
 \calL:= \left\{ L_{d,p} : d \geq 5, d\leq p < 2d\right\}.
 $$
 We have
 
 \begin{theorem}\label{CaseofL} The family $\calL$ is $(10,1,1,5, 9)$--regular.
 
 Furthermore, for $d \geq 5$,  $\calL_d$  respectively   satisfies the condition {\em C1} of Corollary \ref{11/21/4} for $d$ odd or the condition  {\em C2}  for  $d$
  even. 
  
  Finally,  for  every  $d\ge 5$ and for every $\varepsilon >0$,  one has the equality 
 \begin{equation}\label{418}
 \calR_{\geq d} (\calL, B, 0) = \frac 1{(2,d)} \left( \sum_{d\leq p < 2d} A_{L_{d,p}}\right) B^{2/d} + O_{d, \varepsilon} \left(B^{\max\{\vartheta_d, 2/(d+1) \} }
 \right), 
 \end{equation} uniformly for $B  \to\infty$,  
 and the inequalities  
 $$
\frac {\rme^2-o(1)} {\log d} \le \sum_{d\leq p <2d} A_{L_{d,p}}\le \frac {5\,\rme^2+2\,\rme +o(1)} {\log d}
 $$
 uniformly for $d  \to\infty$.
 \end{theorem}
 
  The numerical values are $\rme^2=7.389\dots$ and $5\rme^2+2\rme=42.381\dots$
 
 See \eqref{eta5} for a simplification of the exponent in the error term of \eqref{418}.
 \begin{remark} We now give some hints on the construction  of the family $\calL$.  More generally consider the binary form of degree $d$
 $$L_{{\bf n}, d}(X,Y):= \prod_{1\leq i \leq d} \left(X-n_i Y \right),$$ where ${\bf n} := \{n_1 < n_2  < \cdots <n_d\}$  is a set of $d$ integers.
Fix $d\geq 5$, then for almost all $\bf n$ (in the meaning of Zariski topology), the group of automorphisms of $L_{{\bf n}, d}$ is trivial, which means equal to $\{{\Id}\}$  or to $\{ \pm {\Id} \}$, according to the parity of $d$. Similarly, for fixed $d \geq 5$, for almost all  $({\bf m}, {\bf n})$ the binary forms $L_{{\bf m}, d}$ and $L_{{\bf n}, d}$ are not isomorphic.  For statements of that type, see \cite {Fo1}, for instance. The strategy of choosing $n_1=0$ and  $n_d=p$, where $p$ is   a large prime ensures that the group of automorphisms is trivial and that  the binary forms that we meet are not isomorphic. These statements are proved by appealing to the classical properties of the cross ratio (see \S \ref{6.1} and \S \ref{6.2}).
 
  Finally, we choose for $n_1$,...,$n_{d-1}$ the $d-1$ first integers. This enables us to estimate  the area $A_{L_{d,p}}$  (see \S \ref{areaonL})  via Stirling's formula: 
 \begin{equation}\label{Equation:Stirling}
N^N\rme^{-N}\sqrt{2\pi N}
< N! <
N^N \rme^{-N} \sqrt{2\pi N} \rme^{1/(12N)}, 	
\end{equation}
which is valid for all $N\ge 1$. In particular, as $N \rightarrow \infty$, we have
$$
\log N -1 < \frac 1N \log (N !) < \log N -1 +o(1).
$$
 
 It would be interesting to further investigate the explicit construction of other regular families of forms, which are products of $\Z$--linear forms.
 \end{remark}
 \begin{remark} A natural way to generalize the construction of the families $\calB$,  $\calQ^-$ and $\calQ^+$ is to consider 
 sets of forms which are products of binomials of the shape
 $$ 
 {\rmB}_{a, n} (X,Y)= X^{a} + nY^{a}.
 $$
 The key point is to choose the integers $n$ and the exponents $a\geq 2$ in such a way that we are able to control the homographies in ${\PGl} (2, \Q)$
 which exchange the set of zeroes of the products of ${\rmB}_{a,n}$.
 \end{remark}
 \section{Proof of Theorem \ref{keyresult}}
  \subsection{Beginning of the proof}\label{SS:Beginning}
The starting point is \cite[Th\' eor\`eme 3.1]{FW1}.
 To state this result  we use the following notations:
  
  \noindent $\bullet$ If $F_1$ and  $F_2$ belong to  ${\Bin}(d, \Z)$ and if  $B\geq 1$, we put
  $$
    \begin{aligned}
  {\calM} (F_1 \virgule & F_2;B)
  =
  \\
 & \sharp \left\{(x_1,x_2,x_3,x_4)\in \Z^4 : \max \vert x_i \vert \leq B, F_1(x_1,x_2)= F_2 (x_3,x_4) \right\},
  \end{aligned}
    $$
 and 
    $$
    \begin{aligned}
  {\calM}^* (F_1 \virgule & F_2;B)
  =
  \\
&  \sharp \left\{(x_1,x_2,x_3,x_4)\in \Z^4 : \max \vert x_i \vert \leq B, F_1(x_1,x_2)= F_2 (x_3,x_4)\ne 0 \right\},
  \end{aligned}
  $$
  \noindent $\bullet$ for $d\geq 3$, we introduce 
  \begin{equation*} 
  \eta_d=
  \begin{cases}
  \frac 29 + \frac{73}{108 \sqrt 3} &\text{ for } d=3,\\
  \frac 1{2d} + \frac 9{4d\sqrt d}& \text{ for } 4\leq d\leq 20,\\
  \frac 1d & \text{ for } d\geq 21,
  \end{cases}
  \end{equation*}
  \begin{equation}\label{valueofthetad}
  \vartheta_d = \frac{d\eta_d}{d\eta_d +d-2} 
 \end{equation}
 and
  \begin{equation}\label{eta'=}
  \eta'_{d, F_1,F_2}=
  \begin{cases}
  \eta_d &\text{if  the binary form $F_1(X,Y)F_2(X,Y)$} \\
  &\text{has no zero in  $\PP^1 (\R)$,} \\
  \vartheta_d &\text{otherwise.}
  \end{cases}
  \end{equation}
Here are the first approximate values for $\eta_d$,   $\vartheta_d$ and $\kappa_d$ (recall \eqref{valueofkappad}):
$$
\begin{matrix}
d&\eta_d&\vartheta_d &\kappa_d\\
3& 0.612& 0.647&0.631 \\
4& 0.406& 0.448 &0.428 \hfill\\
5& 0.301& 0.334& 0.309\hfill\\
6& 0.236& 0.261&0.234 \\
7& 0.192& 0.211& 0.184 \\
8&0.161&0.177& 0.150 \\ 
\end{matrix}
$$
  For  $d\geq 3$ and for  $F_1$ and $F_2$ belonging to
  ${\Bin} (d, \Z)$,  one has the inequalities 
 \begin{equation}\label{eta<2/d}
 1/d \leq \eta_d\leq \eta'_{d,F_1,F_2} \leq \vartheta_d<2/d
 \end{equation} 
 and, in particular,  for  $d\geq 21$, we have   : $\eta_d =1/d$ and  $\vartheta_d = 1/(d-1)$. 

Furthermore, by comparison with $\kappa_d$ defined in \eqref{valueofkappad},  we check that   
 \begin{equation}\label{kappa<eta'} 
\begin{cases}
\kappa_d  <\vartheta_d&\text{ if }3\leq d \leq 20,\\
\kappa_d = \vartheta_d &\text{ if } d\geq 21.
\end{cases}
\end{equation} 
Finally, by a direct computation we have the inequalities
\begin{equation}\label{eta'<}
\begin{cases}
\vartheta_{2d} > 1/(d+1) & \text{ if } d=2, \, 3,\\
\vartheta_{2d} <1/(d+1) & \text{ if } d\geq 4,
\end{cases}
\end{equation}
  and 
  \begin{equation}\label{eta5}
  \begin{cases}
  \vartheta_d > 2/(d+1) & \text{ if }  d=4,\, 5,\\
  \vartheta_d < 2/(d+1) & \text{ if } d\geq 6.
\end{cases}
  \end{equation}
 We now recall  (see   \cite [Th\' eor\`eme 3.1]{FW1})
  
\begin{proposition}\label{FouWal} Let  $d \geq 3$ and let $F_1$ and  $F_2$ be two non--isomorphic forms of  ${\Bin}(d, \Z)$,   such that at least one of them is not divisible by a linear form with rational coefficients. Then for all $\varepsilon >0$ and all $B \geq 1$  one has 
  $$ {\calM} (F_1 \virgule  F_2;B) = O_{F_1,F_2,\varepsilon} \left( B^{d\eta_d + \varepsilon} \right).
  $$
  \end{proposition}
As it is shown by  \cite [Remarque 3.2]{FW1}, 
the above bound may not hold if one of the binary forms is divisible by a linear form over $\Q$.  One eliminates this hypothesis by studying the counting function ${\calM^*}$
rather than ${\calM}$. In other words one has  the following variant for Proposition \ref{FouWal}

\begin{proposition}\label{FouWalmod} Let  $d \geq 3$ and let  $F_1$ and  $F_2$ be two non--isomorphic forms of  ${\Bin}(d, \Z)$.   Then for every  $\varepsilon >0$ and
    for all  $B \geq 1$  one has the bound 
  $$ {\calM}^* (F_1 \virgule  F_2;B) = O_{F_1,F_2,\varepsilon} \left( B^{d\eta_d + \varepsilon} \right).
  $$
  \end{proposition}
  
 \begin{proof} We refer  to the original proof of  \cite[Th\' eor\`eme 3.1]{FW1}. The hypothesis that at least one of the
  two forms  $F_1$ and $F_2$ has no $\Q$--linear factor is only used  in  
  \cite[eq. (22)]{FW1} (which is equation (3.8) in the  Arxiv version). This case has no longer to be considered  when one studies ${\calM}^*$ instead of ${\calM}$.   \end{proof}
  \subsection{Lemmas in diophantine approximation }\label{3lemmes}
  Firstly we prove the following
  
\begin{lemma}\label{poly}
  Let  $f \in \Z [t]$  be  a polynomial  of degree $d \geq 1$ and with discriminant different from  zero. Let  $\xi_1, \dots, \xi_d$ be the complex roots of    $f$. 
  Then there are real constants    $c_1>0$ and $c_2$ such that 
   \begin{enumerate}[label=\upshape(\roman*), leftmargin=*, widest=iii]
  \item  
  For every  $t \in \C$, one has the inequality  $ \min_{1\leq j\leq d} \vert t -\xi_j \vert \leq c_2\vert f (t) \vert,$
  \item 
  For  every $t \in \R$, the condition  $\vert f (t) \vert < c_1$ implies  
  the existence of a real root  $\xi_i$  such that   $\vert t -\xi_i\vert \leq c_2 \vert f (t) \vert$.
  \end{enumerate}
  \end{lemma}
  
 \begin{proof}  This statement is trivial when $d=1$.  We now suppose $d \geq 2$.
  
  We suppose that $a_0$ (the leading coefficient  of $f$) is $\geq 1 $ and we factor  $f$ into
  $$
  f(t) = a_0 \prod_{j=1}^d (t -\xi_j).
  $$
  Let  $\delta := \min_{1\leq i < j \leq d} \vert \xi_i- \xi_j\vert$.  Since the discriminant of $f$ is different from zero, we have  $\delta >0$.    
  Let $i$
  be an index such that 
  $\vert t-\xi_i\vert =\min_{1\leq j \leq d} \vert t -\xi_j\vert$.   The triangular inequality  gives, for $j \ne i$,  the lower bound 
  $$
  \vert t -\xi_j \vert \geq \frac{\vert t-\xi_j\vert + \vert t -\xi_i\vert}2 \geq \frac 12\, \vert \xi_j - \xi_i \vert \geq \frac \delta 2\cdotp
  $$
  We write the sequence of inequalities
  $$
  \vert f (t) \vert \geq \prod_{1\leq j \leq d} \vert t -\xi_j\vert \geq \vert t -\xi_i\vert \left( \frac \delta 2 \right)^{d-1},
  $$
 which leads to the point {\rm (i)} with $c_2 = (2/\delta)^{d -1}. $
   
   For the item {\rm (ii)}, we now suppose that  $t$ is real.  We decompose the proof into three cases.  
   
   \vskip .3cm
   \noindent $\bullet$  If all the  $\xi_j$ are real, there is nothing to prove  as a consequence of  {\rm (i)}. We choose $c_1=1$ for instance.
   \vskip .3cm
   \noindent $\bullet$ If no  $\xi_j$ is real, we set
   $$
   c_1:= \inf_{x \in \R} \vert f (x) \vert
   $$
   which is $>0$.
   \vskip .3cm
   \noindent $\bullet$ If  $f$  has at least one real root and at least one non real root, we put 
   $$
   c_1 = \frac 1 {c_2} \min\bigl\{ \,\vert {\Im} (\xi_i) \vert : 1\leq i \leq d, \, \xi_i \not\in \R\, \bigr\}.
   $$
 Applying the item  {\rm (i)} of Lemma \ref{poly},  we notice that for $t\in\R$ the inequality   $\vert f(t) \vert < c_1$ implies the existence of  a root  $\xi_j$
   such that
   $$
   \vert t- \xi_j \vert <  c_1 c_2 = \min\bigl\{ \,\vert {\Im} (\xi_i-t) \vert : 1\leq i \leq d, \, \xi_i \not\in \R\, \bigr\}.
   $$
   If  $\xi_j$ were not real, we would deduce the inequality $\vert t -\xi_j \vert < \vert {\Im}(t-\xi_j )\vert$, which is impossible. 
   Hence  $\xi_j$ is real.
  \end{proof} 
 The following lemma provides an upper bound for the tail of the series defining the Riemann $\zeta$--function. 
   
\begin{lemma}\label{zeta} For all  real $\delta >1$ and all positive integer  $B$, one has the inequality
  $$
  \sum_{n\geq B} \frac 1{n^\delta} \leq \zeta (\delta) B^{1-\delta}.
  $$
  \end{lemma}
 
 \begin{proof}  By dividing the interval of summation in intervals with length $B$ and by using the inequality 
 $Bq+r\geq Bq$, we write 
 $$
 \sum_{n\geq B} \frac 1 {n^\delta} = \sum_{q\geq 1} \sum_{r=0}^{B-1} \frac 1 {(Bq+r)^\delta} \leq B^{1-\delta} \sum_{q\geq 1} \frac 1{q^\delta} =\zeta (\delta)\, B^{1-\delta}.
 $$
 \end{proof} 
The next lemma was inspired by   \cite[p. 34--36]{Ho1}.  
  
\begin{lemma}\label{InspirHooley} Let  $\xi$, $\kappa$, $s$, $Q_1$ and  $Q_2$ be real numbers such that  $s>2$, 
  $\kappa >0$, $Q_2 >Q_1 \geq 1$. Then the number of rational numbers  $\frac pq$ such that
  $$
  \Bigl\vert \xi -\frac pq\Bigr\vert \leq \frac \kappa{q^s} \text{  and  } Q_1\leq q \leq Q_2,
  $$
  is bounded   by  
 $$ 
 \frac{2^{s+1}\kappa }{(2^{s-2}-1) Q_1^{s-2}} +\left\lceil{ \frac{\log  \frac{Q_2}{Q_1}}{\log 2}}\right\rceil.
  $$
  \end{lemma}
 
 \begin{proof}  Firstly we consider the case when $Q_2 \leq 2 Q_1$  and we prove the result with the coefficient  $\frac {2^{s+1}}{2^{s-2} -1}$ replaced by $8$. 
 Two distinct  rational numbers   $\frac pq$, $\frac{p'}{q'}$ such that $Q_1\leq q, q' \leq Q_2$ satisfy the inequalities
 $$
 \left\vert \frac pq - \frac {p'}{q'}\right\vert \geq \frac 1{qq'} \geq \frac 1 {Q_2^2} \geq \frac 1{4 Q_1^2}\cdotp
 $$
 If  they also satisfy
 $$
   \Bigl\vert \xi -\frac pq\Bigr\vert \leq \frac \kappa{q^s}  \text{ and }   \Bigl\vert \xi -\frac {p'}{q'}\Bigr\vert \leq \frac \kappa{{q'}^s},
   $$
then they belong to the interval
$$
\left[ \xi -\frac{\kappa}{Q_1^s}, \xi +\frac \kappa{Q_1^s}\right],
$$ 
 the length of which is  $2\kappa/Q_1^s$. So the number of  such $\frac pq$ is less than  
$$
4Q_1^2 \frac {2\kappa}{Q_1^s} +1 =\frac {8\kappa}{Q_1^{s-2}} +1.
$$
In the case where  $Q_2 >2Q_1$, we cover the interval  $[Q_1,Q_2]$ by $\ell$ intervals $[2^h Q_1, 2^{h+1} Q_1]$, $0\leq h\leq \ell -1$ with 
$2^{\ell -1} Q_1 < Q_2 \leq 2^\ell Q_1$;  thus  $\ell$ satisfies the inequalities 
$$
 \frac {\log \frac{Q_2}{Q_1}}{\log 2}\leq \ell  < 1 + \frac {\log \frac{Q_2}{Q_1}}{\log 2}\cdotp
$$
As we have seen, in the interval   $[2^h Q_1, 2^{h+1} Q_1]$, the number of rational numbers $\frac pq$ satisfying our assumption is bounded by
$$
\frac {8\kappa} {2^{h(s-2)} Q_1 ^{s-2}} +1.
$$
The total number of fractions $\frac pq$ satisfying our assumption  is less than 
$$
\sum_{h =0}^{\ell -1}\left(  \frac {8\kappa} {2^{h(s-2)} Q_1 ^{s-2}} +1\right) = \frac {8\kappa}{ Q_1^{s-2}}\sum_{h=0}^{\ell -1} \frac 1{2^{h (s-2)}}  +\ell  < \frac {8\kappa}{Q_1^{s-2}} \cdot \frac 
{2^{s-2}}{2^{s-2}-1} + \left\lceil{  \frac{ \log \frac{Q_2}{Q_1}}{\log 2} }\right\rceil.
$$
 \end{proof} 
\subsection{On the set of the values taken by a binary form when one of the variables is large } \label{2.3} As a consequence of the three lemmas proved in   \S \ref{3lemmes}  
we will deduce
\begin{proposition} \label{Prop2.3} Let $d \geq 3$ and let  $F \in {\Bin} (d, \Z)$. Then there are two constants  $c_3$ and  $c_4$,  effectively
computable  and depending on  $F$ only, such that, for all $\Delta >c_3$ and all $A >0$ one has the following inequality
$$
\sharp \left\{ (x,y)\in \Z^2 :
0 < \vert F (x,y)\vert \leq A, \vert y \vert \geq A^{1/d} \Delta
\right\}
\leq c_4 \left(
A^{2/d} \Delta^{2-d} + A^{1/(d-1)}
\right).
$$
\end{proposition}
The proof of this proposition will use the following  effective refinement of Liouville's inequality, due to   N. I. Fel'dman~\cite{F}:

\begin{lemma}\label{Feldman} Let  $\xi$ be an algebraic number of degree  $d\geq 3$.  There are two effectively computable positive constants
 $c_5= c_5 (\xi)$ and  $c_6 = c_6 (\xi)$   such that, for every fraction  $p/q\in \Q$  with $q \geq 1 $, one has the inequality 
$$
\left\vert \xi -\frac pq \right\vert \geq \frac{c_5}{q^{d-c_6}}\cdotp
$$
\end{lemma}
A completely explicit version of this inequality  can be found in  \cite[(13) p. 248]{GP}.  

We deduce from this lemma the following one.

\begin{lemma}\label{Feldmangeneralise}  Let  $P(X)\in \Z [X]$ be a  polynomial,  of degree $d\geq 3$.  There are two effectively computable positive constants  
$c'_5 = c'_5 (P)$ and  $c'_6=c'_6 (P)$ such that,
for every root $\xi$ of  $P$,  for every rational number  $p/q$ such that  $q\geq  1$ and   $p/q\ne \xi$, the following inequality holds 
\begin{equation}\label{412}
\left\vert \xi -\frac pq
\right\vert \geq \frac {c'_5}{q^{d-c'_6}}\cdotp
\end{equation}
\end{lemma} 
We stress that there is no assumption on whether the polynomial $P$ is irreducible or not,  nor on whether the root $\xi$ is  real or not.

 \begin{proof} Let $\delta$ be the degree of $\xi$. We split the argument according to the value of $\delta$ and to the nature of $\xi$.
\begin{itemize}
\item If  $\xi$ is not real, the inequality   \eqref{412}  is trivial since we have  $\vert \xi -p/q\vert \geq | {\Im}\, \xi|$, for every rational number  $p/q$.
\vskip .3cm
 We now suppose that  $\xi$ is a real number. 
 \vskip .3cm

\item If $\delta =1$. We put  $\xi =a/b$ with  $a$ and  $b$ integers and  $b\geq 1$. We have  $\vert a/b -p/q\vert = \vert aq-bp \vert /  bq \geq 1/ bq$, since  $\xi$ is different from  $p/q$. We obtain  \eqref{412}, with the choices $c'_5 =1/b$ and $c'_6=1$ since $d \geq 3$.
\item If $ \delta =2$. The real number  $\xi$ is  quadratic. Liouville's inequality for quadratic real numbers is optimal: there exists   $\alpha =\alpha (\xi) >0$ such that
$$
\left\vert
\xi -\frac pq \right\vert \geq \frac {\alpha }{q^2}\cdotp
$$
By the hypothesis $d\geq 3,  $  we deduce  \eqref{412} with the choice $c'_5 =\alpha $ and   $c'_6 =1/2$. 
\item If $\delta \geq 3$. We apply Lemma  \ref{Feldman} in the form 
$$
\left\vert  \xi -\frac pq  \right\vert \geq  \frac{c_5 }{q^{\delta -c_6 }}\cdotp
$$
Since  $\delta \leq d$, we obtain  \eqref{412} with the choice  $c'_5= c_5$ and  $c'_6 =c_6$. 
\end{itemize}
We  choose for  $c'_5= c'_5 (P)$ and for  $c'_6= c'_6 (P)$ the least values $c'_5$ and $c'_6$  corresponding to the various $\xi$  that we met above to complete the proof  of Lemma \ref{Feldmangeneralise}.
\end{proof}
\vskip .3cm
\noindent  {\em Proof of Proposition  \ref{Prop2.3}}. 
  Let  $f(t) =F(t,1)$, thus we have $F(x,y) =y^d f(x/y)$. Let  $d'$ be the degree of $ f$.  Since the discriminant of $F$ is different from zero, we have 
    $$
  d' =d \text{ or } d-1.
  $$ 
    If  $f$ has no real root, then, for sufficiently large  $\Delta$ (more precisely, for $\Delta > \left( \inf_{t\in \R} \vert f(t) \vert \right)^{-1/d}$), the set 
  $$
  \left\{ (x,y)\in \Z^2 :
0 < \vert F (x,y)\vert \leq A, \vert y \vert \geq A^{1/d} \Delta
\right\}
$$ is empty. 

Let  $r\geq 1$  be the number of real roots of $f$, that we denote by $\xi_1$,..., $\xi_r$. By hypothesis these roots are simple.  Let  $(x,y) \in \Z^2$ with $y\ne 0$. The condition $0< \vert F( x,y)\vert \leq A$ implies 
$$
0 < \left\vert f\left( \frac xy\right)\right\vert \leq \frac A{\vert y \vert^d}\cdotp
$$
We suppose  $\vert y \vert \geq A^{1/d} \Delta$ and  $\Delta> c_1^{-1/d}$,  and we apply   Lemma \ref{poly} {\rm (ii)}.
We deduce the existence of some  $i\in \{ 1, \dots, r\}$ such that 
 \begin{equation}\label{(1.1)}
 0< \left\vert \frac xy-\xi_i \right\vert \leq \frac {c_2A}{\vert y \vert^d},
 \end{equation}
 which is equivalent to 
 \begin{equation}\label{(1.2)}
 0< \vert x-y\xi_i\vert \leq \frac{c_2 A}{\vert y \vert^{d-1}}\cdotp
 \end{equation}
When the integer  $y$ is fixed, the number of integers  $x$ satisfying the inequality  \eqref{(1.2)} is equal to 
 $$
 \frac{2c_2A}{\vert y \vert^{d-1}} +O (1).
 $$
We fix  $Y_0=A^{1/(d-1)}$ and we sum over  $i=1, \dots, r$.  We apply   Lemma \ref{zeta} with  $B= A^{1/d} \Delta$ and  $\delta = d-1$,  to deduce that the number of 
   $(x,y)$ with  $0 < \vert F (x,y) \vert \leq A$ and  $A^{1/d} \Delta \leq \vert y \vert \leq Y_0$  is bounded by 
   \begin{equation}\label{(1.3)}
    O\bigl(A^{2/d}\Delta^{2-d} \bigr) +O (Y_0).
 \end{equation}

 To complete the proof, we use      Lemma \ref{Feldmangeneralise} which implies the lower bound  
 \begin{equation}\label{376}
 \left\vert \xi_i-\frac pq\right\vert \geq \frac{c'_5}{q^{d'-c'_6}}\geq \frac{c'_5}{q^{d-c'_6}}\cdotp
\end{equation}
  Combining    \eqref{(1.1)} with \eqref{376}, we deduce the upper bound   $\vert y \vert \leq Y_1$ with $Y_1 =(\frac {c_2}{c'_5} A)^{1/{c'_6}}$.
It remains to compute the number of solutions of \eqref{(1.1)} satisfying  $Y_0< \vert y \vert \leq Y_1$. We use  Lemma  \ref{InspirHooley}, with  $s= d$, $\kappa = c_2 A$, $Q_1= Y_0$, $Q_2=Y_1$: this  number is bounded by  
 $$
 O\left(A/Y_0^{d-2} \right) + O\left( \log Y_1 \right)= O \left( A^{1/(d-1)}\right).
 $$
By adding  \eqref{(1.3)} we obtain the upper bound announced in  Proposition \ref{Prop2.3}.
  \hfill$\square$ 
 
  \subsection{End of the proof of  Theorem  \ref{keyresult}}\label{2.4} We split the end of the proof into two different cases:
  \vskip .3cm
  \noindent $\bullet$ Assume the binary form  $F_1 (X,Y) F_2 (X,Y)$ has no zero in  $\PP^1 (\R)$.
  This hypothesis holds true if and only if   the polynomial  $F_1(t,1)F_1 (1,t)F_2(t,1) F_2(1,t)$  has no real root. By homogeneity, there is a constant  $c_7>0$ such that for all 
  $(x_1, x_2, x_3, x_4)\in \R^4$, one has the inequalities 
  $$
  \vert F_1 (x_1, x_2) \vert \geq c_7 \max \{ \vert x_1\vert^d, \vert x_2 \vert^d \} \text{ and  }   \vert F_2 (x_3, x_4) \vert \geq c_7 \max \{ \vert x_3\vert^d, \vert x_4 \vert^d \}.
  $$
 This leads to  the existence of a constant  $c_8$ such that the inequalities  $$\vert F_1 (x_1,x_2) \vert \leq N\text{ and  } \vert F_2 (x_3,x_4) \vert \leq N$$
  imply
 $\max(
  \vert x_1\vert, \vert x_2 \vert, \vert x_3\vert, \vert x_4 \vert ) \leq B$ with  $B:= (c_8 N)^{1/d}$. 
We apply  Proposition~\ref{FouWalmod} under the form 
$$
{\calN}(F_1\virgule F_2;N)\leq 1 + {\calM}^*(F_1\virgule F_2; B) = O_{F_1,F_2} \left(B^{d\eta_d +\varepsilon}\right) = O_{F_1,F_2} \left( N^{\eta_d +\varepsilon}\right).
$$
By the inequality   \eqref{eta<2/d},  the proof of Theorem \ref{keyresult} is complete in that case, including the refinement quoted in Remark \ref{Remark:NoRealRoot}.
\vskip .3cm
\noindent $\bullet$  Assume the binary form $F_1 (X,Y) F_2 (X,Y)$ has  at least one zero in $\PP^1 (\R)$.  This is equivalent
to the assumption that the polynomial 
$$
F_1(t,1)F_1 (1,t)F_2(t,1) F_2(1,t)
$$ 
has at least one real root.  The constant $\eta'_{d,F_1,F_2}$ is now defined by the second formula of 
  \eqref{eta'=}, that is $\eta'_{d,F_1,F_2}=\vartheta_d$. Let  
$$
\tau := \frac {\frac 2d -\eta_d}{d\eta_d+d-2},
$$
so we have the equalities 
$$
\frac 2d -(d-2)\tau= \eta_d (1+d \tau) = \eta'_{d,F_1,F_2}.
$$
Let  $\Delta := N^\tau$. To bound from above ${\calN} (F_1\virgule F_2; N)$,  which is the number of  $m\in \Z$, $\vert m \vert \leq N,$  such that there is  at least one  $(x_1,x_2,x_3,x_4)\in \Z^4$ satisfying
the equality
\begin{equation}\label{F=F=m}
F_1 (x_1,x_2) = F_2 (x_3, x_4) =m,
\end{equation}
we first consider those $m$ such that at least one of $(x_1,x_2,x_3,x_4)$ associated to $m$ by \eqref{F=F=m}  satisfies  the inequality
$$
\max\{ \vert x_1\vert, \vert x_2\vert, \vert x_3\vert , \vert x_4\vert \} < N^{1/d} \Delta.
$$
Proposition \ref{FouWalmod} with $B= N^{\frac 1d +\tau}$ shows that the number of these $m$  is bounded by
\begin{equation}\label{422}
O_{F_1,F_2,\varepsilon} \left( B^{d\eta_d +\varepsilon} \right) = O_{F_1,F_2,\varepsilon} \left( N^{\eta_d(1+d\tau)+\varepsilon}\right)= 
O_{F_1,F_2,\varepsilon} \left(N^{\eta'_{d,F_1,F_2} + \varepsilon}\right).
\end{equation}
Next, we estimate the number of  those $m$  such that all  the $4$--tuples  $(x_1,x_2,x_3,x_4)$ associated to $m$ by \eqref{F=F=m}  satisfy  the inequality
$$
\max\{ \vert x_1\vert, \vert x_2\vert, \vert x_3\vert , \vert x_4\vert \} \geq  N^{1/d} \Delta.
$$
For simplicity, we study the case where  $\vert x_1\vert \geq N^{1/d} \Delta$, since the other cases are similar. We only consider the values taken by the binary form  $F_1$ and we apply
 Proposition \ref{Prop2.3}. With the choices   $F=F_1$ and  $A=N$, using $\vartheta_d\ge 1/(d-1)$, we  deduce that the number of corresponding  $m$ is bounded by 
$$
O_{F_1,F_2} \left( N^{2/d} \Delta^{2-d} +N^{1/(d-1)}\right)=O_{F_1,F_2}\left( N^{\eta'_{d,F_1,F_2}}\right).
$$
By  \eqref{422}, this completes the proof of  Theorem  \ref{keyresult}.

 \section{Proof of Theorem \ref{Thm-on-R}} 
 By similarity with   \eqref{defRgeqd}, we put 
\begin{multline*}
{\calR}_{=d}( \calF,B,A ):= \sharp \{m: 0\leq \vert m \vert \leq B, \text{ there is  }   F\in \calF_{d}, (x,y)\in \Z^2, \\
\text{ such that  } \max(\vert x \vert, \vert y \vert ) \geq A\text{ and }m =F(x,y)
\}.
\end{multline*}
The lower bound for  ${\calR}_{\geq d}( \calF,B,A )$ is obtained as follows
\begin{align*}
{\calR}_{\geq d}( \calF,B,A ) &\geq  {\calR}_{=d}( \calF,B,A ) \\
& \geq \sum_{F\in \calF_d}  {\calN}( F\virgule F;B)- \underset{\atop{F,\ F'\in \calF_d }{ F \ne F'}} {\sum \ \sum} {\calN}(F\virgule F';B)
- (2A+1)^2 d^{A_1},
\end{align*}
where the counting function  ${\calN}$ is defined by  \eqref{defNF1=F2}. Condition 
(iii) in Definition \ref{agreable} of a  regular family implies $\sharp \calF_d=O_d(1)$;  thanks to condition 
(iv) and  to the inequality $\kappa_d \leq \vartheta_d$ (see \eqref{kappa<eta'}),  Theorems  \ref{keyresult} and  \ref{Stewart} give the inequality 
\begin{equation}\label{m0}
{\calR}_{\geq d}( \calF,B,A )
\geq \left( \sum_{F\in \calF_d} A_{F } W_F\right)\cdot B^{2/d} -O_{\calF,A,\varepsilon} \left( B^{\vartheta_d +\varepsilon} \right).
\end{equation}

For the upper bound, we recall that the parameters $d_0, \, \kappa$ and $A_1$  appear in Definition \ref{agreable}. We start from the inequality
\begin{multline}
{\calR}_{\geq d}( \calF,B,A ) \leq \sum_{F\in \calF_d} {\calN} (F\virgule F;B) 
+ \sum_{n=d^\dag}^{d^\dag +d_0} \sum_{F \in \calF_n} {\calN} (F\virgule F; B) \\ + \sharp\, \left( \underset{n> d^\dag +d_0}{\bigcup}\, \underset{ F\in \calF_n}{\bigcup} \left( F(\calZ_A) \cap [-B,B] \right) \right),\label{m1}
\end{multline}
with $$
\calZ_A = \Z^2\setminus \bigl( [-A, A]\times [-A,A]\bigr).
$$
Applying one more time Theorem  \ref{Stewart},  we have the equality 
\begin{equation}\label{m2} 
\sum_{F\in \calF_d} {\calN} (F\virgule F;B) = \left( \sum_{F\in \calF_d} A_{F } W_F\right)\cdot B^{2/d} +O_{\calF,d,\varepsilon}  \left( B^{\kappa_d +\varepsilon} \right),
\end{equation}
and the upper bound
\begin{equation}\label{m3}
{\calN} (F\virgule F;B) = O_F\left( B^{2/d^\dag}\right)\text { if } \deg F \geq d^\dag.
\end{equation}
Hence the second term on the right--hand side of \eqref{m1} is bounded as follows
$$
 \sum_{n=d^\dag}^{d^\dag +d_0} \sum_{F \in \calF_n} {\calN} (F\virgule F; B) =
 O_{\calF,d} \left( B^{2/d^\dag}\right).
$$
To deal with the third term on the right--hand side of \eqref{m1}, we interchange the summations to write
\begin{multline}\label{657}
\sharp\, \left( \underset{n> d^\dag +d_0}{\bigcup}\, \underset{ F\in \calF_n}{\bigcup} \left( F(\calZ_A) \cap [-B,B] \right) \right)  \\
\leq \sharp\, \left\{ (n, F, x, y): n > d^\dag +d_0, F \in \calF_n, (x,y) \in \calZ_A,  \vert F(x,y) \vert \leq B\right\}. 
\end{multline}
The  condition 
(v) in Definition   \ref{agreable}  of the 
$(A,A_1, d_0, d_1, \kappa)$--regularity  of  $\calF$ produces a bound for  $n$,  by the sequence of inequalities
\begin{equation}\label{660}
\kappa<A \leq \max\{\vert x\vert, \vert y \vert \}\leq \kappa \vert F(x,y)\vert^{\frac 1{n-d_0}}\leq \kappa B^{\frac 1{n-d_0}}\leq \kappa B^{\frac 1{d^\dag +1}},
\end{equation}
which implies the inequality
$$
n\leq d_0 + \frac{\log B}{\log (A/\kappa)}\cdotp
$$
Furthermore the inequalities  \eqref{660} implies 
$$
\max \{ \vert x \vert, \vert y \vert \} \leq \kappa B^{ 1/(d^\dag +1)}.
$$
Combining the above inequalities, we deduce that the cardinality of the quadruples  $(n, F,x,y)$  
in the right--hand side of  \eqref{657} is bounded from above by 
\begin{equation}\label{m4}
\left( d_0 + \frac{\log B}{\log (A/\kappa)}\right)^{A_1 } \left( 1 +2 \kappa B^{1/({d^\dag +1})}\right)^2 = o_\calF \left( B^{   2/{d^\dag}}\right). 
\end{equation}
Gathering  \eqref{m1}, \eqref{m2}, \eqref{m3} and  \eqref{m4}, we finally obtain the upper bound 
\begin{equation}\label{m5}
{\calR}_{\geq d}( \calF,B,A ) \leq \left( \sum_{F\in \calF_d} A_{F } W_F\right)\cdot B^{2/d} +O_{\calF,A,d,\varepsilon} \left( B^{\kappa_d+ \varepsilon } \right) +
O_{\calF,d}\left( B^{2/d^\dag}\right).
\end{equation}
Comparing  \eqref{m0} and \eqref{m5} and recalling the inequality \eqref{kappa<eta'},  we complete the proof of Theorem~\ref{Thm-on-R}.
\section{Proof of Theorem  \ref{CaseofQ+}}

\subsection{The family $\calQ^+$ is $(2,1,0,4,1)$--regular} 
Our first purpose is to prove the following

\begin{proposition}\label{4.1}
The family $\calQ^+$ is $(2,1, 0, 4, 1)$--regular.
\end{proposition}

 \begin{proof} Several times, we will use the following property satisfied by two positive  distinct squarefree numbers
\begin{equation}\label{algebra1}
n\ne n' \Rightarrow \Q\left( i\sqrt{ {\mu}_n}\right) \ne  \Q\left( i\sqrt{ {\mu}_{n'}}\right).
\end{equation}
We now check   each of the items
of Definition \ref{agreable} of a regular family.

$\bullet$ The items {\rm (i)} and {\rm (ii)} are trivial.

$\bullet $ The family $\calQ^+$ contains no element with odd degree $d$. By contrast if this degree $d \geq 4$ is even, the family contains $d/2 +1$ binary forms of degree $d$.

Thus the item {\rm (iii)} is verified with $A_1 =1$.

$\bullet $ For the item {\rm (iv)} we proceed as follows.  Suppose that there are two isomorphic  forms $F$ and $F'$ in $\calQ^+$. Necessarily they have the same degree
$2d$. So there exist $1\leq \nu < \nu'\leq d+1$ and a matrix $\gamma \in {\Gl}(2, \Q)$, written as in \eqref{defgamma},  such that
$$
Q_{d, \nu}^+ =Q_{d, \nu'}^+ \circ \gamma.
$$
Let $\tilde \gamma$ be the homography attached to $\gamma$. This homography 
\begin{equation}\label{gammatilde}
\tilde \gamma:z\in \PP^1(\C)\mapsto \frac{a_1z+a_2}{a_3z+a_4}
\end{equation}
induces a bijection   between the set of zeroes $\calZ (Q_{d, \nu}^+)$ (in $\PP^1 (\C)$)  of $Q_{d, \nu}^+$
and the set of zeroes $\calZ (Q_{d, \nu'}^+)$. So, $\tilde \gamma (i\sqrt {{\mu}_{\nu'}})$  is  a zero of $Q_{d, \nu'}^+$, hence is one of  $\pm i\sqrt {{\mu}_n}$ with $n\not=\nu'$, which contradicts \eqref{algebra1}.

$\bullet $ The definition \eqref{defQ+} implies that $Q_{d, \nu}^+ (x,y)=0$ if and only if $(x,y)=(0,0)$. Furthermore, by positivity, we have
the lower bound 
$$
\left\vert Q_{d, \nu}^+ (x,y)\right\vert \geq \left(\max\{ \vert x\vert^2, \vert y \vert^2\}
\right)^d = \left( \max \{ \vert x \vert, \vert y \vert \} \right)^{\deg Q_{d,\nu}^+}.
$$
The above inequality implies
$$
\max\{ \vert x \vert, \vert y \vert \}\leq   \left\vert Q_{d, \nu}^+ (x,y)\right\vert^{ 1/{\deg Q_{d,\nu}^+}},
$$
which means the item {\rm (v)} is satisfied for $A=2$,  $d_0 =0$, $d_1=4$  and $\kappa =1$.
\end{proof}

\subsection{Triviality of the group  ${\Aut} ( Q_{d, \nu}^+,  \Q)$}  
We now prove
\begin{proposition}\label{4.2} For every $d\geq 2$ and $1\leq \nu \leq d+1$, one has
$$
{\Aut} (Q_{d, \nu}^+, \Q) =\left\{\begin{pmatrix} \pm 1 &0\\
0 & \pm 1
\end{pmatrix}\right\} 
$$
(Klein group of order $4$). 
\end{proposition}

 \begin{proof} The four elements 
$$
\begin{pmatrix}  1 &0\\
0 &  1
\end{pmatrix} ,\quad\begin{pmatrix} - 1 &0\\
0 & - 1
\end{pmatrix} ,\quad\begin{pmatrix} - 1 &0\\
0 &  1
\end{pmatrix} ,\quad
\begin{pmatrix}  1 &0\\
0 & -1
\end{pmatrix}
$$
in ${\Gl}(2, \Q)$ clearly belong to ${\Aut} (Q_{d, \nu}^+, \Q) $. 
Conversely, let $\gamma=\begin{pmatrix}a_1&a_2\\a_3&a_4\end{pmatrix} \in {\Gl}(2, \Q)$ and let $Q_{d, \nu}^+$ be such that
\begin{equation}\label{Q=Q}
Q_{d,\nu}^+ \circ \gamma = Q_{d, \nu}^+.
\end{equation}
The set of zeroes $ \calZ (Q_{d,\nu}^+)$ is stable by the homography $\tilde \gamma$ attached to $\gamma$. Appealing to \eqref{algebra1}, we deduce
$$
\tilde \gamma \left(i \sqrt{{\mu}_n}\right) = \varepsilon_n i  \sqrt{{\mu}_n}, ( 1\leq n \leq d+1, \, n\ne \nu),
$$
where $\varepsilon_n = \pm 1$. We now prove that the value of  $\varepsilon_n$ is independent from $n$. Indeed, suppose that there exist $m$ and $n$ 
such that $\varepsilon_m =1$ and $\varepsilon_n =-1$. Returning to the explicit expression of $\tilde \gamma$ (see \eqref{gammatilde}), we obtain $$
\begin{cases}
a_1i \sqrt{{\mu}_m }+a_2 = i \sqrt{{\mu}_m} \left( a_3 i \sqrt{{\mu}_m} +a_4 \right)
 \\ 
 a_1i \sqrt{{\mu}_n }+a_2  = -i \sqrt{{\mu}_n} \left( a_3 i \sqrt{{\mu}_n} +a_4 \right).\\
\end{cases}
$$
Since $a_1$, $a_2 $, $a_3$ and $a_4$ are rational numbers, we deduce the four  equalities
$$
\begin{cases}
a_2 = -a_3\, {\mu}_m\\
a_2 = a_3 \, {\mu}_n\\
 a_1 =a_4\\
 a_1=-a_4.
\end{cases}
$$
These equalities imply  $(a_1,a_2 ,a_3,a_4)=(0,0,0,0)$ which is forbidden. So  we have $\tilde \gamma (z) = \varepsilon  z$, for some fixed $\varepsilon \in \{ \pm 1\}$. This means that for some $\tau \in \Q$, we have
$$\gamma = \begin{pmatrix} \varepsilon  \tau &0\\
0 &   \tau
\end{pmatrix} .
$$
By identification in \eqref{Q=Q}, we find that $\tau = \pm1$.
\end{proof}

\subsection{Estimating the number  of  images  by  $\calQ^+$ of $(x,y)$ with $\max \{ \vert x \vert, \vert y \vert\} \ge 2$} 
For the family $\calQ^+$, one has $(2d)^\dag =2d+2$. Combining Corollary \ref{11/21/4}, Propositions~\ref{4.1} and  \ref{4.2} and the equality  \eqref{W=1W=1}, we proved the following

\begin{proposition}\label{4.3} For every $d\geq 2$, one has the equality 
$$
\mathcal R_{\geq 2d} (\mathcal Q^+, B,2) =\frac 14 \left( \sum_{F\in \mathcal Q_{2d}^+} A_F\right) \cdot B^{1/d} + O_{\lambda,d,\varepsilon}\left(B^{\vartheta_{2d} +\varepsilon}\right) +O_{\lambda,d}\left( B^{1/(d+1)}\right).
$$
\end{proposition}
\subsection{Estimating the number of images by $\mathcal Q^+$ of $(x,y)$ with $\max \{ \vert x \vert, \vert y \vert\} < 2$}  \label{smallQ+} 
The difference
\begin{equation}\label{differenceQ+}
\mathcal R_{\geq 2d} (\mathcal Q^+, B,0)-\mathcal R_{\geq 2d} (\mathcal Q^+, B,2)
\end{equation}
is bounded from above by  the cardinality of the set
\begin{multline}\label{boundingd'}
\{ m\; :\; 0\le m\leq B,\, m= Q_{d', \nu}^+ (\pm 1, \pm 1),\,  d' \geq d, \, 1\leq \nu  \leq d'+1\} \\ 
\cup \{ m\; :\; 0\le m\leq B,\, m= Q_{d', \nu}^+
 (0, \pm 1),\,  d' \geq d, \, 1\leq \nu  \leq d'+1\}
\cup \{0, 1\}.
\end{multline}
For every $d'$ and $1\leq \nu \leq d'+1$, one has the equality
$$ Q_{d', \nu}^+(\pm 1, \pm 1)
\geq \prod_{1\leq n \leq d'}\left( 1+n^2\right) \geq (d'\, !)^2.
$$
This implies that the inequality $Q_{d', \nu}^+ (\pm 1, \pm1 )\leq B$ can only hold if $d' =O \left (\log B\right)$. So  the cardinality of the first set in \eqref{boundingd'} is bounded by 
$O(\log^2 B)$. The same bound also applies to the second set. Combining Proposition  \ref{4.3} with \eqref{differenceQ+} we obtain

\begin{proposition}\label{4.4} For every $d\geq 2$ and for every $\varepsilon >0$,  one has the equality 
$$
\calR_{\geq 2d} (\calQ^+, B,0) =\frac 14 \left( \, \sum_{F\in \calQ_{2d}^+} A_F\right)  \cdot B^{1/d}+ O_{ \lambda, d,\varepsilon}\left(B^{\vartheta_{2d} +\varepsilon}\right) +O_{\lambda,d}\left( B^{1/(d+1)}\right).
$$
\end{proposition}
\subsection{Some results on $A_F$ for $F\in \calQ^+$} By the definition \eqref{defAF}, the fundamental domain attached to $Q_{d, \nu}^+$ is
\begin{equation}\label{924}
\calD (Q_{d, \nu}^+):=\bigl\{ (x,y) \in \R^2 : \prod_{\atop{1\leq n\leq d+1}{ n \ne \nu}} \left (x^2 +{\mu}_n y^2\right) \leq 1 \bigr\}.
\end{equation}
Our purpose is to estimate the sum
$$
{\Coef}(\calQ^+, 2d):= \sum_{F\in \calQ_{2d}^+} A_F
 $$
 as $d\to\infty$.
We use integration techniques to express this sum of fundamental areas as follows.

\begin{lemma}\label{Fubini}
For any $d \geq 2$  and $1\le \nu\le d+1$, one has the equality
$$
 A_{Q_{d,\nu}^+} = \int_{-\infty}^\infty \frac{(u^2+\mu_\nu)^{1/d}}{G_d(u)^{1/d}}  \rmd  u,
 $$
 where 
$$ G_d(u) :=\prod_{n=1}^{ d+1} \left( u^2 + \mu_n\right).
$$ 
Hence
  $$
{\Coef}(\calQ^+, 2d) = \int_{-\infty}^{\infty} \frac
{\sum_{1\leq  n \leq d+1} \left( u^2+\mu_n\right)^{1/d}}
{G_d(u)^{1/d}}
\, {\rmd} u.
 $$
 \end{lemma} 
 
 \begin{proof} By \eqref{924} and by the change of variables $x=uv$, $y=v$ we have the equalities
 \begin{align*}
 A_{Q_{d,\nu}^+} & = \iint_{\calD (Q_{d,\nu}^+)} {\rmd}x \, {\rmd} y\\
 & = \iint_{ v^{2d}\prod_{\atop{1\leq n\leq d+1}{  n\ne \nu}} (u^2+\mu_n)\leq 1}  \vert v\vert \, {\rmd} u\, {\rmd} v\\
 & = \int_{-\infty}^\infty \frac{{\rmd}  u}{\prod_{\atop{1\leq n\leq d+1}{ n\ne \nu}} (u^2+\mu_n)^{1/d}}\cdotp
 \end{align*}
 Compare with \cite{B} p.~122. 
 Summing over all the $Q_{d, \nu}^+\in \calQ_{2d}^+$, we obtain the second formula of Lemma \ref{Fubini}.
 \end{proof}
We first give a lower bound of ${\Coef}(\calQ^+, 2d)$. We have
\begin{align}
{\Coef}(\calQ^+, 2d) & \geq  (d+1)  \int_{-\infty}^\infty \frac{(u^2 +\mu_1)^{1/d}}{G_d(u)^{1/d}} {\rmd } u\nonumber\\
& \geq  (d+1) \int_{-\infty}^\infty  \frac {{\rmd} u}{\prod_{2\leq n \leq d+1} (u^2+\mu_{n})^{1/d}}\nonumber\\
&\geq (d+1) \int_{-\infty} ^{\infty} \frac{{\rmd} u}{u^2+\mu_{d+1}}\nonumber\\
&\geq \pi\cdot \frac{d+1}{\sqrt{\mu_{d+1}}}\cdotp\label{--}
\end{align}
From our assumption $\mu_{d+1} \le \lambda(d+1)$ we deduce from  \eqref{--}  the  lower bound
\begin{equation}\label{987}
{\Coef}(\calQ^+, 2d) > \frac \pi  {\sqrt \lambda}\, \sqrt d. 
\end{equation}
For the upper bound, we write 
\begin{align*}
{\Coef}(\calQ^+, 2d) & \leq  (d+1)  \int_{-\infty}^\infty \frac{(u^2 +\mu_{d+1})^{1/d}}{G_d(u)^{1/d}} {\rmd } u\\
& \leq  (d+1) \int_{-\infty}^\infty  \frac {{\rmd} u}{\prod_{1\leq n \leq d} (u^2+\mu_{n})^{1/d}}\cdotp
\end{align*}
Using  H\" older's  inequality we deduce 
\begin{align*}
{\Coef}(\calQ^+, 2d) & \leq    (d+1)\prod_{n=1}^d \left( \int_{-\infty}^\infty  \frac {{\rmd} u}{ u^2+\mu_{n}}\right)^{1/d}
\\
&
\leq \pi (d+1)\prod_{n=1}^d \mu_{n}^{-1/(2d)}\le \pi\frac{d+1}	D
\end{align*}
with $D:= (d\, !)^{1/(2d)}$. Using Stirling formula \eqref{Equation:Stirling},  
we deduce  
$$
{\Coef}(\calQ^+, 2d) \leq  \pi   \sqrt{\rme} (\sqrt d +1).
$$
Combining with \eqref{987} we complete the proof of \eqref{sqrt6}.  Recalling Propositions 
\ref{4.2}  and \ref{4.4},  we conclude that the proof of Theorem \ref{CaseofQ+} is now complete.
\section{Proof of Theorem \ref{CaseofQ-}}

Recall that for $d\ge 2$ and $1\le \nu\le d+1$,  $Q^-_{d,\nu} $ denotes the following binary form of degree  $2d$
$$
 Q^-_{d,\nu} (X,Y)
=\prod_{\atop{1\leq n \leq d+1}{n\ne \nu}} \left( X^2 - {\mu}_n Y^2\right)
$$
and $ \calQ^-$  denotes the family
$$
 \calQ^-=\left\{ Q^-_{d, \nu} : d\geq 2,\, 1\leq \nu \leq d+1\right\}.
 $$

\subsection{The family $\calQ^-$ is $(A,1,2,2,2\rme\lambda)$--regular} 

Our goal in this subsection  is to prove the following

\begin{proposition}\label{5.1}
For $A>2\rme\lambda$, the family $\calQ^-$ is $(A,1,2,2,2\rme\lambda)$--regular. 
\end{proposition}

The proofs of items (i), (ii), (iii) and (iv) are the same as for Proposition \ref{4.1}: one only replaces \eqref{algebra1} with the remark that  for two positive  distinct squarefree numbers
$$
n\ne n' \Rightarrow \Q\left( \sqrt{ {\mu}_n}\right) \ne  \Q\left( \sqrt{ {\mu}_{n'}}\right).
$$
It remains to check the condition 
 {\rm (v)} in Definition \ref{agreable} of a regular family. 
We start with an auxiliary lemma.

\begin{lemma}\label{Lemme:minoration}
For $m$ and $d$ integers satisfying $1\le m<d$, we have
$$
\left(\frac d m -1\right)^{d-m}\ge \rme^{- \rme^{-1} m };
$$
further, for  $n$ an integer in the range $1\le n\le d$, we have
$$
\frac{n!(d-n)!}{n^d}\ge \rme^{-(1+\rme^{-1})d}.
$$
\end{lemma}

 The numerical value for $\rme^{1+\rme^{-1}}$ is $3.927\dots<\frac {79}{20}\cdotp$

\begin{proof}[Proof of Lemma \ref{Lemme:minoration}]
Set $t=d-m$, $f_m(t)=\left(\frac t m \right)^{t}$, $g_m(t)=\log f_m(t)=t\log t -  t\log m$. The derivative $g_m'(t)=1+\log t  -\log m$ of $g_m$  vanishes at  $t=m/\rme$, the minimum of  $f_m(t)$ on the interval $0<t\le d-1$ is reached at $t=m/\rme$, giving the value $(t/m)^t=\rme^{-t}=\rme^{-m/\rme}$.

The last part of Lemma \ref{Lemme:minoration} follows from the first one thanks to  Stirling's formula  \eqref{Equation:Stirling}:
$$
\frac{n!(d-n)!}{n^d}\ge \frac{n^n}{\rme^n}\cdot \frac{(d-n)^{d-n}}{\rme^{d-n}}\cdot \frac1 {n^d}
= \frac{ (d-n)^{d-n}}{n^{d-n} \rme^d }
\ge 
 \rme^{-d}\rme^{-\rme^{-1}n}\ge 
 \rme^{-(1+\rme^{-1})d}.
$$
\end{proof}

The  last inequality of  Lemma \ref{Lemme:minoration}  implies
\begin{equation}\label{Equation:consequenceLemme:minoration}
\bigl(n!(d-n)!\bigr)^{1/d} \ge    \rme^{-1-\rme^{-1}}\max\{n,d-n\}\ge \frac d {2\rme^{1+\rme^{-1}}}\ge \frac d {2\rme^{2}}\cdotp
\end{equation}

 \begin{proof}[End of the proof of Proposition \ref{5.1}]
 Let $d\ge 2$,   $1\le \nu\le d+1$,  $(x,y)\in\Z^2\setminus\{(0,0)\}$. Set $Q=Q^-_{d,\nu} (x,y)$. Our goal it to prove
 \begin{equation}\label{equation:minorationQ-}
  |Q| > (2\rme \lambda)^{-2d+2} \max\{|x|,|y|\}^{2d-2}.
   \end{equation}
  
  We consider three cases depending on the sign of the factors $x^2-\mu_ny^2$. 
  \\
  $\bullet$ If $x^2< \mu_1y^2$, all factors are negative.  For $2\le n\le d+1$ we have
  $$
  |x^2 - {\mu}_n y^2|={\mu}_n y^2  -x^2  > ({\mu}_n -\mu_1)y^2.
  $$
 When $\nu\ge 2$, we use  the lower bound $\mu_1y^2-x^2\ge 1$ and obtain 
  $$
  |Q|> (\mu_2- \mu_1)\cdots(\mu_{d+1}-  \mu_1)(\mu_\nu-  \mu_1)^{-1}y^{2d-2}\ge (d-1)! y^{2d-2}.
  $$
  For $\nu=1$ the stronger lower bound $|Q|> (d-1)! y^{2d}$ holds. 
  Hence  for $1\le \nu\le d+1$ we have
  $$
  |Q| > \frac{(d-1)! }{\mu_1^{d-1}} \, x^{2d-2}\ge \frac{(d-1)! }{\lambda^{d-1}} \, x^{2d-2}.
  $$
  The desired estimate \eqref{equation:minorationQ-} follows.
  \\
  $\bullet$ If $x^2> \mu_{d+1}y^2$, all factors are positive and $\max\{|x|,|y|\}=|x|$. For $m=1,\dots,d$ we have
  $$
  x^2-\mu_my^2 > (\mu_{d+1}- \mu_m) \frac{x^2}{\mu_{d+1}},
  $$
while for  $m=d+1$ we have  $x^2- \mu_{d+1}y^2\ge 1$. Hence  for $1\leq \nu \leq d$ we have 
 $$
  |Q| =Q> 
 (\mu_{d+1}- \mu_1)
 (\mu_{d+1}- \mu_2)
    \cdots
 (\mu_{d+1}- \mu_d)
 (\mu_{d+1}- \mu_\nu)^{-1} \frac{x^{2d-2}}{\mu_{d+1}^{d-1}}\ge \frac{d!x^{2d-2}}{\mu_{d+1}^{d-1}}\cdotp
      $$
     The lower bound  $|Q|> \frac{d!x^{2d-2}}{\mu_{d+1}^{d-1}}$ is also true when $\nu=d+1$ since in this case we have
   $$
  |Q| =Q> 
 (\mu_{d+1}- \mu_1)
 (\mu_{d+1}- \mu_2)
    \cdots
 (\mu_{d+1}- \mu_d) \frac{x^{2d}}{\mu_{d+1}^{d}}  \ge \frac{d!x^{2d}}{\mu_{d+1}^{d}} 
 $$   
 and $x^2> \mu_{d+1}y^2\ge \mu_{d+1}$.
Since $d!>d^d\rme^{-d}$ (see  \eqref{Equation:Stirling}) and  $\mu_{d+1}\le \lambda(d+1)$, we have 
$$
\frac {d!}{\mu_{d+1}^{d-1}}> \frac{d}{\rme \left(1+\frac 1 d\right)^{d-1}(\lambda\rme)^{d-1}}>
\frac 1{ (2\rme\lambda)^{2d-2}}\cdotp
$$
This implies  \eqref{equation:minorationQ-}.
 \\
  $\bullet$ Finally, assume that there is an $n$ in the interval $1\le n\le d$ such that 
  $$
  x^2 - \mu_{n+1}y^2<0<x^2 - \mu_ny^2.
  $$ 
Hence $y\not=0$ and  $\max\{|x|,|y|\}=|x|$. We have
  \begin{align}\label{Q1}
  &|Q|=\\
\notag
  &(x^2-\mu_1y^2)(x^2-\mu_2y^2)\cdots(x^2-\mu_ny^2)(\mu_{n+1}y^2-x^2)\cdots (\mu_{d+1}y^2-x^2)|x^2-\mu_{\nu}y^2|^{-1}
  \end{align}
with
\begin{align}\label{Q2}
(x^2-\mu_1y^2)(x^2-\mu_2y^2)\cdots(x^2-\mu_{n-1}y^2)&>
(\mu_n-\mu_1)(\mu_n-\mu_2)\cdots(\mu_n-\mu_{n-1}) y^{2n-2}
\\
\notag
& \ge (n-1)! y^{2n-2} 
\end{align}
and
\begin{align}\label{Q3}
(\mu_{n+2}y^2-x^2)\cdots (\mu_{d+1}y^2-x^2) &>
(\mu_{n+2}-\mu_{n+1})\cdots (\mu_{d+1}-\mu_{n+1}) y^{2d-2n}
\\
\notag
&
\ge(d-n)! y^{2d-2n}.
\end{align}
For $1\le \nu\le n-1$, we use the lower bound
\begin{equation}\label{Q4}
(x^2-\mu_1y^2)(x^2-\mu_2y^2)\cdots(x^2-\mu_{n-1}y^2)(x^2-\mu_\nu y^2)^{-1} > (n-2)! y^{2n-4},
\end{equation}
while for $ n+2\le \nu\le d+1$, we use the lower bound
\begin{equation}\label{Q5}
(\mu_{n+2}y^2-x^2)\cdots (\mu_{d+1}y^2-x^2) (\mu_\nu y^2-x^2)^{-1}>  (d-n-1)! y^{2d-2n-2}.
\end{equation}
It remains to estimate the product   $
(x^2-\mu_ny^2)(\mu_{n+1}y^2-x^2) $ of the two terms of the middle in \eqref{Q1}. We consider two cases.
\par
$\diamond$ 
 Assume $|y|\ge 2$.  If $\nu\in\{n,n+1\}$, we use the trivial lower bound
\begin{equation}\label{Q6}
(x^2-\mu_ny^2)(\mu_{n+1}y^2-x^2) |x^2-\mu_\nu y^2|^{-1}\ge 1,
\end{equation}
while if $\nu\le n-1$ or $\nu\ge n+2$ we use the lower bound
\begin{align}\label{Q7}
(x^2-\mu_ny^2)(\mu_{n+1}y^2-x^2)
&
\ge(x^2-\mu_ny^2)+(\mu_{n+1}y^2-x^2)-1
\\
\notag
&=
 (\mu_{n+1}-\mu_n)y^2-1 \ge y^2-1
\ge \frac 3 4 y^2.
\end{align}
$\circ$
For  $\nu\in\{n,n+1\}$, we deduce from \eqref{Q1}, \eqref{Q2}, \eqref{Q3}, \eqref{Q6}, 
$$
|Q|\ge (n-1)!(d-n)! y^{2d-2}.
$$
$\circ$
For $1\le \nu\le n-1$,  we deduce from 
\eqref{Q1},  \eqref{Q3}, \eqref{Q4}, \eqref{Q7}, 
$$
|Q|\ge \frac 3 4  (n-2)!(d-n)! y^{2d-2}.
$$ 
$\circ$
For $ n+2\le \nu\le d+1$ we deduce from 
\eqref{Q1}, \eqref{Q2},   \eqref{Q5},  \eqref{Q7},  
$$
|Q|\ge   \frac 3 4  (n-1)!(d-n-1)! y^{2d-2}.
$$ 
In the three cases, namely for $1\le\nu\le d+1$, we have, thanks to Lemma  \ref{Lemme:minoration}, 
$$
|Q|\ge   \frac {3 n!(d-n)!}{4  n(d-1)} y^{2d-2}\ge \frac {3n^{d-1}}{4(d-1)\rme^{(1+\rme^{-1})d}} y^{2d-2}.
$$
From $x^2< \mu_{n+1}y^2\le \lambda (n+1) y^2\le 2\lambda n y^2$ we deduce 
$$
|Q| >  \frac 3 {4(d-1)\rme^{(1+\rme^{-1})d}(2\lambda)^{d-1}} x^{2d-2}.
$$
Finally, since  $\lambda\ge  2$, we have 
\begin{equation}\label{avec4/3}
\frac 3{4(d-1) \rme^{(1+\rme^{-1})d}   } 
 >  \frac 1 { (2\rme^2 \lambda)^{d-1}}
\end{equation}
 for $d\ge 2$, and  \eqref{equation:minorationQ-} follows, 
\par
$\diamond$ If $y^2=1$, hence $\mu_n<x^2<\mu_{n+1}$,
using the trivial lower bound
$$
(x^2-\mu_n)(\mu_{n+1}-x^2)  \ge 1,
$$
and a combination of the above lower bounds
\eqref{Q1}, \eqref{Q2}, \eqref{Q3}, \eqref{Q4}, \eqref{Q5}
yields
$$
|Q|\ge \begin{cases}
(n-1)!(d-n)! & \hbox{if $\nu\in\{n,n+1\}$},
\\
(n-2)!(d-n)! & \hbox{if $1\le \nu\le n-1$},
\\
(n-1)!(d-n-1)! & \hbox{if $n+2\le \nu\le d+1$}.
\end{cases}
$$
For $1\le \nu\le d+1$, we obtain, thanks to Lemma  \ref{Lemme:minoration}, 
$$
|Q|\ge \frac{n!(d-n)!}{n(d-1)}\ge \frac {n^{d-1}}{(d-1)\rme^{(1+\rme^{-1})d}}\cdotp
$$
If $x^2\le 2\lambda$, using  $n\ge 1$, we deduce
$$
|Q|\ge   \frac {n^{d-1}}{(d-1)\rme^{(1+\rme^{-1})d}} \left( \frac{x^2}{2\lambda}\right)^{d-1},
$$
while if $x^2\ge 2\lambda$ we have, by \eqref{mun<}, the inequalities  $n> \frac {x^2}\lambda-1\ge \frac{x^2}{2\lambda}$, hence again
$$
|Q|\ge  \frac {x^{2d-2}}{(d-1) \rme^{(1+\rme^{-1})d}   (2\lambda)^{d-1}} \cdotp
$$
From \eqref{avec4/3} we deduce the estimate  \eqref{equation:minorationQ-}  also when $|y|=1$.
\par 
This completes the proof of Proposition \ref{5.1}.
\end{proof}

\subsection{Triviality of the group  ${\Aut} ( Q_{d, \nu}^-,  \Q)$}  

The following result is the analog of Proposition \ref{4.2}. The proof is the same, since $\mu_1\geq 2$ and the roots of $Q_{d, \nu}^-$ are all irrational numbers.  

\begin{proposition}\label{5.2} For every $d\geq 2$ and $1\leq \nu \leq d+1$, one has
$$
{\Aut} (Q_{d, \nu}^-, \Q) =\left\{\begin{pmatrix} \pm 1 &0\\
0 & \pm 1
\end{pmatrix}\right\} 
$$
(Klein group of order $4$). 
\end{proposition}

\subsection{Estimating the number  of images by   $\calQ^-$ of $(x,y)$ with $\max \{ \vert x \vert, \vert y \vert\}~\ge A$} 
From Corollary \ref{11/21/4}, the equality  \eqref{W=1W=1} and Propositions \ref{5.1} and \ref{5.2}, we deduce: 

\begin{proposition}\label{5.3}
For every $A>2\rme\lambda$,  for every $d\geq 2$ and for every $\varepsilon >0$, one has the equality
$$
\calR_{\geq 2d} (\calQ^-, B,A) =\frac 14 \left(\,  \sum_{F\in \calQ_{2d}^-} A_F\right) \cdot B^{1/d} + O_{\lambda,A,d, \varepsilon}\left(B^{\vartheta_{2d} +\varepsilon}\right) +O_{\lambda,A,d}\left( B^{1/(d+1)}\right).
$$
\end{proposition}

\subsection{Estimating the number  of images by  $\calQ^-$ of $(x,y)$ with $\max \{ \vert x \vert, \vert y \vert\}~< A$} 
The difference  
$$
\calR_{\geq 2d} (\calQ^-, B,0) -\calR_{\geq 2d} (\calQ^-, B,A)
$$
is at most the cardinality of the set 
$$
\left\{m\; :\; 0\ne  |m|\le B,\,
m= Q_{d',\nu}^-(x,y),\,  d'\ge d,\,  1\le \nu\le d'+1, \, \max\{|x|,|y|\}\le A
\right\}.
$$
Given $d'$, the number of such $m$ in this set  is bounded by $(d'+1)(2A+1)^2$. Hence we only need to bound from above the value of $d'$ when $|m|\ge 2$.

We first consider the integers of the form $Q_{d',\nu}^-(x,0)$. Since $Q_{d',\nu}^-(\pm 1,0)=1$, we may assume $|x|\ge 2$. From 
$$
 Q_{d',\nu}^-(x,0)=x^{2d'}\le B
 $$
 we deduce that $d'$ is bounded by $O(\log B)$. 
 
 Next let $m=Q_{d',\nu}^-(x,y)$ with $d'\ge d$, $1\le \nu\le d'+1$, $\max\{|x|,|y|\}\le A$, $|y|\ge 1$ and $0<|m|\le B$. 
Without loss of generality we may assume $d'>2A^2$.  We split the product
 $$
 \prod_{\atop{1\leq n\leq d'+2}{ n \ne \nu}} \left |x^2 -{\mu}_n y^2\right| 
 $$
  the value of which is $|m|$, as $P_1P_2$ where 
 $$
 P_1=\prod_{\atop{1\leq n\leq 2A^2}{ n \ne \nu}} \left |x^2 -{\mu}_n y^2\right|,
 \qquad
 P_2=\prod_{\atop{2A^2<  n\leq d'+2}{ n \ne \nu}} \left |x^2 -{\mu}_n y^2\right|.
 $$
 The product $P_1$ is $\ge 1$. For $2A^2<  n\leq d'+2$, since $\mu_n>n$,  $|x|\le A$ and $|y|\ge 1$, we have $ \mu_n y^2-x^2 \ge A^2$, hence 
 $$
 (A^2)^{d'-2A^2}\le P_2\le P_1P_2=|m|\le B,
 $$
 which yields 
 $$
 d'\le 2A^2+\frac{\log B}{2\log A}=O_{A}(\log B).
 $$
 Hence 
 $$
\calR_{\geq 2d} (\calQ^-, B,0) -\calR_{\geq 2d} (\calQ^-, B,A)
=O_{A}((\log B)^2).
$$
Thanks to Proposition \ref{5.3}, this completes the proof of the estimate for $\calR_{\geq 2d} (\calQ^-, B, 0 )$ in Theorem
\ref{CaseofQ-}.  
\subsection{Some results on $A_F$ for $F\in \calQ^-$} By the definition \eqref{defAF}, the fundamental domain attached to $Q_{d, \nu}^-$ is
\begin{equation}\label{924Bis}
\calD (Q_{d, \nu}^-):=\bigl\{ (x,y) \in \R^2 : \prod_{\atop{1\leq n\leq d+1}{ n \ne \nu}} \left |x^2 -{\mu}_n y^2\right| \leq 1 \bigr\}.
\end{equation}
Our purpose is to estimate the sum
$$
{\Coef}(\calQ^-, 2d):= \sum_{F\in \calQ_{2d}^-} A_F
 $$
as $d\to\infty$ by proving \eqref{averageAFQ-}. 

Repeating the proof of  Lemma \ref{Fubini}, we obtain:
 
\begin{lemma}\label{FubiniBis}
For any $d \geq 2$ and $1\le \nu\le d+1$,  one has the equality
  $$
  A_{Q_{d,\nu}^-}  = \int_{-\infty}^\infty \frac{ |u^2-\mu_\nu|^{1/d}}{\prod_{1\leq n\leq d+1}  |u^2-\mu_n|^{1/d}} {\rmd}  u.
  $$ 
  Hence 
  $$
  {\Coef}(\calQ^-, 2d) = \int_{-\infty}^\infty \frac{
  \sum_{1\le n\le d+1}  |u^2-\mu_n|^{1/d} }{\prod_{n=1}^{d+1}  |u^2-\mu_n|^{1/d}} {\rmd}  u.
  $$
 \end{lemma} 
 
 Since $ |u^2-\mu_n|\le  u^2+\mu_n$, the lower bound of ${\Coef}(\calQ^-, 2d)$ is a consequence of the lower bound of ${\Coef}(\calQ^+, 2d)$. More precisely, we have, by Lemma \ref{FubiniBis},
$$
A_{Q_{d,\nu}^-}  = \int_{-\infty}^\infty \frac{{\rmd}  u}{\prod_{\atop{1\leq n\leq d+1}{ n\ne \nu} }|u^2-\mu_n|^{1/d}}
\ge 
 \int_{-\infty}^\infty \frac{{\rmd}  u}{\prod_{\atop{1\leq n\leq d+1}{ n\ne \nu}} (u^2+\mu_n)^{1/d}},
$$
hence
\begin{align*}
{\Coef}(\calQ^-, 2d) & \geq  (d+1)   \int_{-\infty}^\infty  \frac {{\rmd} u}{\prod_{2\leq n \leq d+1} (u^2+\mu_{n})^{1/d}}\\
&\geq (d+1) \int_{-\infty} ^{\infty} \frac{{\rmd} u}{u^2+\mu_{d+1}}
= \pi\cdot \frac{d+1}{\sqrt{\mu_{d+1}}}\cdotp
\end{align*}
This  proves the  lower bound
\begin{equation}\label{987Bis}
{\Coef}(\calQ^-, 2d) > \frac\pi {\sqrt \lambda}   \, \sqrt d. 
\end{equation}

For the upper bound,  we use once more Lemma \ref{FubiniBis}. By the change of variable $u^2=v$ we have
$$
A_{Q_{d,\nu}^-}  = 
2\int_0^\infty \frac{{\rmd}  u}{\prod_{\atop{1\leq n\leq d+1}{ n\ne \nu} }|u^2-\mu_n|^{1/d}}
=
\int_0^\infty \frac{{\rmd}  v}{\sqrt v\prod_{\atop{1\leq n\leq d+1}{ n\ne \nu} }|v-\mu_n|^{1/d}}\cdotp
$$
We split the integral as the sum of $d+3$ terms
$$
A_{Q_{d,\nu}^-}  = \sum_{j=0}^{d+2} A_j
$$
with
$$
A_0=\int_0^{\mu_1}
 \frac{{\rmd}  v}{\sqrt v\prod_{\atop{1\leq n\leq d+1}{ n\ne \nu}} (\mu_n-v)^{1/d}},
$$
$$
A_j=
\int_{\mu_j}^{\mu_{j+1}}
 \frac{{\rmd}  v}{\sqrt v\prod_{\atop{1\leq n\leq j}{ n\ne \nu} }(v-\mu_n)^{1/d}\prod_{\atop{j+1\leq n\leq d+1}{ n\ne \nu} }(\mu_n-v)^{1/d}}
\qquad (1\le j\le d+1)
$$
and
$$
A_{d+2}=\int_{\mu_{d+2}}^\infty  \frac{{\rmd}  v}{\sqrt v\prod_{\atop{1\leq n\leq d+1}{ n\ne \nu} }(v-\mu_n)^{1/d}}\cdotp
$$
$\bullet$ Upper bound for $A_0$.
\\
For $\nu=1$, we use the lower bound
$$
\prod_{ 2\leq n\leq d+1 } (\mu_n-\mu_1) \ge d!\ge \frac {d^d}{e^d}
$$
which follows from Stirling's estimate \eqref{Equation:Stirling} and 
one deduces
$$
A_0\le
\frac1 {
\prod_{ 2\leq n\leq d+1 } (\mu_n-\mu_1)^{1/d}}
 \int_0^{\mu_1}
 \frac{{\rmd}  v}{\sqrt v} \le \frac{2\rme\sqrt{\mu_1}}d \le \frac{2\rme\sqrt{\lambda}}d\cdotp
$$
Similarly, for $2\le \nu\le d+1$ we have
$$
A_0\le
\frac1 {
\prod_{\atop{2\leq n\leq d+1}{ n\ne \nu}} (\mu_n-\mu_1)^{1/d}}
 \int_0^{\mu_1}
 \frac{{\rmd}  v}{\sqrt v (\mu_1-v)^{1/d}} 
$$
and
$$
\prod_{\atop{2\leq n\leq d+1}{ n\ne \nu}} (\mu_n-\mu_1) \ge  (d-1)! = \frac{ d! }{d},
$$
hence
$$
\prod_{\atop{2\leq n\leq d+1}{ n\ne \nu}} (\mu_n-\mu_1)^{1/d}\ge   \frac d {\rme\sqrt 2}.
$$
From the upper bounds (recall $\lambda\ge 2$ and $2\le \mu_1\le \lambda$)
$$
\begin{aligned}
 \int_0^{\mu_1}
 \frac{{\rmd}  v}{\sqrt v (\mu_1-v)^{1/d}}
 &\le \int_0^{\mu_1}  \frac{{\rmd}  v}{\sqrt v }
 +
  \int_0^{\mu_1}  \frac{{\rmd}  v}{ (\mu_1-v)^{1/d}}
  \\
  &=2\sqrt {\mu_1}+\frac d {d-1} \mu_1^{1-( 1 /d)} < (2+\sqrt 2)
  \lambda,
  \end{aligned}
$$ 
we deduce  
$$
A_0< \frac {5 \rme   \lambda}{d}\cdotp
$$ 
 \\
$\bullet$ Upper bound for $A_j$, $1\le j\le d+1$. 
\\
$\diamond$ If $\nu\not\in\{j,j+1\}$, we have
$$
\begin{aligned}
A_j&\le  
\frac 1
{\sqrt \mu_j\prod_{\atop{1\leq n\leq j-1}{ n\ne \nu} }(\mu_j-\mu_n)^{1/d}\prod_{\atop{j+2\leq n\leq d+1}{ n\ne \nu} }(\mu_n-\mu_{j+1})^{1/d}
}
\\ 
&\qquad \qquad \qquad \qquad \qquad \qquad \qquad \qquad \qquad 
\int_{\mu_j}^{\mu_{j+1}}
 \frac{{\rmd}  v}{  (v-\mu_j)^{1/d} (\mu_{j+1}-v)^{1/d}}\cdotp
 \end{aligned}
$$
We use \eqref{Equation:consequenceLemme:minoration}:  for $1\le j\le d$ we have 
$$
\begin{aligned}
\prod_{\atop{1\leq n\leq j-1}{ n\ne \nu} }(\mu_j-\mu_n)
\prod_{\atop{j+2\leq n\leq d+1}{ n\ne \nu} } (\mu_n-\mu_{j+1})
&\ge
\begin{cases}
\frac{(j-1)!(d-j)!}{j-\nu} &\hbox{for $1\le \nu\le j-1$}
\\
\frac{(j-1)!(d-j)!}{\nu-j-1} &\hbox{for $j+1\le \nu\le d+1$}
\end{cases} 
\\
&\ge \frac {j!(d-j)!}{d^2}\ge \frac 1 {d^2} \left( \frac d {2\rme^{1+\rme^{-1}}}\right)^d,
 \end{aligned}
 $$
 while for $j=d+1$ this lower bound becomes
 $$
 \prod_{\atop{1\leq n\leq d}{ n\ne \nu} }(\mu_{d+1}-\mu_n)\ge \frac{d!}{d+1-\nu}
 \ge \frac 1 {d^2} \left( \frac d {2\rme^{1+\rme^{-1}}}\right)^d.
 $$
Next we use the following estimate:
$$ 
\begin{aligned}
&\int_{\mu_j}^{\mu_{j+1}}
 \frac{{\rmd}  v}{  (v-\mu_j)^{1/d} (\mu_{j+1}-v)^{1/d}}
 \le
 \\ 
 &  \qquad  
 \frac {2^{1/d}} {  (\mu_{j+1}-\mu_j)^{1/d}}
 \left(
 \int_{\mu_j}^{(\mu_j+\mu_{j+1})/2}
 \frac{{\rmd}  v}{  (v-\mu_j)^{1/d}}
 + 
 \int_{(\mu_j+\mu_{j+1})/2}^{\mu_{j+1}} \frac{{\rmd}  v}{  (\mu_{j+1}-v)^{1/d}}\right).
 \end{aligned}
 $$ 
 We have
 $$ 
 \int_{\mu_j}^{(\mu_j+\mu_{j+1})/2}
 \frac{{\rmd}  v}{  (v-\mu_j)^{1/d}}
 =
  \int_{(\mu_j+\mu_{j+1})/2}^{\mu_{j+1}} \frac{{\rmd}  v}{  (\mu_{j+1}-v)^{1/d}}
=\frac d {d-1} \left(\frac{\mu_{j+1}-\mu_j} 2\right)^{1-(1/d)}.
 $$ 
 Hence 
 $$
 \int_{\mu_j}^{\mu_{j+1}}
 \frac{{\rmd}  v}{  (v-\mu_j)^{1/d} (\mu_{j+1}-v)^{1/d}}\le
 \frac {d} {d-1} 2^{2/d } (  \mu_{j+1}-\mu_j)^{1-(2/d) }.
 $$
 We deduce that for $\nu\not\in\{j,j+1\}$, we have 
 $$
 A_j\le (4d^2)^{1/d} 2 \rme^{1+\rme^{-1}}  \frac {(\mu_{j+1}-\mu_j)^{1-(2/d)}}{(d-1) \sqrt{\mu_j}} \cdotp
 $$
 $\diamond$ If $\nu=j$, we have 
 $$ 
A_j\le  
\frac 1
{\sqrt \mu_j\prod_{1\le n\le j-1} (\mu_j-\mu_n)^{1/d}\prod_{j+2\le n\le d+1}(\mu_n-\mu_{j+1})^{1/d}
} 
\int_{\mu_j}^{\mu_{j+1}}
 \frac{{\rmd}  v}{  (\mu_{j+1}-v)^{1/d}}
$$
and we use the formula
 $$
 \int_{\mu_j}^{\mu_{j+1}}  \frac{{\rmd}  v}{  (\mu_{j+1}-v)^{1/d}}
 =
 \frac d {d-1} (\mu_{j+1}-\mu_j)^{1-(1/d)}.
 $$
$\diamond$  If $\nu=j+1$, we have
$$
A_j\le  
\frac 1
{\sqrt \mu_j\prod_{1\le n\le j-1} (\mu_j-\mu_n)^{1/d}\prod_{j+2\le n\le d+1}(\mu_n-\mu_{j+1})^{1/d}
}
 \int_{\mu_j}^{\mu_{j+1}}  \frac{{\rmd}  v}{  (v-\mu_j)^{1/d}}
$$
and we use the formula
 $$
 \int_{\mu_j}^{\mu_{j+1}}  \frac{{\rmd}  v}{  (v-\mu_j)^{1/d}}
 =
 \frac d {d-1} (\mu_{j+1}-\mu_j)^{1-(1/d)}.
$$ 

\par
We deduce that for $1\le j\le d+1$ and $1\le \nu\le d+1$, we have
 \begin{equation}\label{Ajleq}
A_j\le \bigl(2 \rme^{1+\rme^{-1}}+o(1)\bigr) \frac { \mu_{j+1}-\mu_j} {d\sqrt{\mu_j}}
\cdotp
\end{equation} 
For $j\ge 1$, we have
$$
\mu_j \geq j+1 \geq \frac 1\lambda \, \mu_{j+1}, 
$$
and we deduce the inequality 
$$
\sum_{j=1}^{d+1}  \frac { \mu_{j+1}-\mu_j} {\sqrt{\mu_j}}
\le \sqrt {\lambda} \, 
\sum_{j=1}^{d+1}  \frac { \mu_{j+1}-\mu_j} {\sqrt{\mu_{j+1}}}\cdotp  
$$
Using  the inequality
$$
\sum_{j=1}^{d+1} \frac { \mu_{j+1}-\mu_j}  {\sqrt {\mu_{j+1}}} 
\le 
\sum_{j=1}^{d+1}\int_{\mu_j}^{\mu_{j+1} }\frac{{\rmd}  t}{ \sqrt t}=
\int_{\mu_1}^{\mu_{d+2}} \frac{{\rmd}  t}{ \sqrt t}
\le 
 2\sqrt {\mu_{d+2}}\le2 \sqrt {\lambda(d+2)},
$$
we deduce from \eqref{Ajleq}, that 
\begin{align*}
\sum_{j=1}^{d+1} A_j \le &   \left( \bigl(2 \rme^{1+\rme^{-1}} +o_\lambda(1)\bigr)/d\right) \cdot \sqrt \lambda \cdot  \bigl( 2\sqrt{\lambda (d+2)}\bigr)\\
& \leq   \bigl(4 \rme^{1+\rme^{-1}} +o_\lambda(1)\bigr) \frac {\lambda}{\sqrt d}\cdotp
\end{align*}

$\bullet$ Upper bound for $A_{d+2}$. 
\\
For $v\ge \mu_{d+2}$ and $1\le n\le d+1$, we have
$$
v-\mu_n\ge v\left(1-\frac{\mu_n}{\mu_{d+2}}\right),
$$
hence
$$
A_{d+2}\le
\frac 1
{\prod_{\atop{1\leq n\leq d+1}{ n\ne \nu} }\left(1-\frac {\mu_n}{\mu_{d+2}}\right)^{1/d}}
\int_{\mu_{d+2}}^\infty  \frac{{\rmd}  v}{v^{3/2}}
$$
with 
$$
\int_{\mu_{d+2}}^\infty  \frac{{\rmd}  v}{v^{3/2}}=\frac 2 {\sqrt{\mu_{d+2}}}\le\frac 2 {\sqrt{ d+2}}
$$
and (using Stirling's estimate  \eqref{Equation:Stirling} once more)
$$
\prod_{\atop{1\leq n\leq d+1}{ n\ne \nu} }\left(1-\frac {\mu_n}{\mu_{d+2}}\right)^{1/d}
\ge 
\frac {d!^{1/d}} {\mu_{d+2}}
 \ge 
\frac {d!^{1/d}} {\lambda(d+2)}  
\ge 
\frac 1{\lambda \rme } \left(1+\frac 2 d\right)^{-1}.
$$
We deduce  
$$
A_{d+2}\le \left(   2 \rme    +o(1)\right) \frac \lambda  {\sqrt d}\cdotp
$$

Putting these estimates together, we obtain  
$$
A_{Q_{d,\nu}^-} = \sum_{j=0}^{d+2} A_j  \le \left( 4 \rme^{1+\rme^{-1}}  +  2 \rme  +o_\lambda(1)\right) \frac \lambda  {\sqrt d}\cdotp
$$
Summing over all the $Q_{d, \nu}^-\in \calQ_{2d}^-$ we conclude
 $$
{\Coef}(\calQ^-, 2d) \le \left( 4 \rme^{1+\rme^{-1}} + 2 \rme+o_\lambda(1)\right) \lambda  \sqrt d.
$$
Combining with \eqref{987Bis} and  with the upper bound  $4 \rme^{1+\rme^{-1}} + 2 \rme<22$, we complete the proof of \eqref{averageAFQ-}. The proof of Theorem \ref{CaseofQ-} is now complete.

\section {Proof of Theorem \ref{CaseofL}}  We now use the notations of \S\, \ref{1.3.3}.  Our first purpose is to check that the family   $\calL$ satisfies the assertions of Definition \ref{agreable} of a regular family. 
 The items {\rm (i)}, {\rm (ii)} are obvious. The item {\rm (iii)} is trivially satisfied with $A_1=1$. The items {\rm (iv)} and {\rm (v)} 
 are more subtle. 
 
 \subsection{Isomorphisms between two elements in $\calL$} \label{6.1} We will prove the following more general statement which implies that the item {\rm (iv)} is fulfilled by  the family $\calL$.
 
\begin{proposition}\label{1026} 
 Let $d\geq 4$ be an integer,  $\{a_i: 1\leq i \leq d-1\}$   and $\{b_j : 1\leq j \leq d-2\}$  two sets of integers and $p$  a prime number such that
 \begin{equation}\label{firstinequality}
 0<a_1<\cdots <a_{d-1} <p,  
 \end{equation}
 and 
 \begin{equation}\label{secondinequality} 
 0 <b_1<\cdots <b_{d-2} <p.
 \end{equation}
 Then the binary forms
 \begin{equation}\label{defFG}
 X\prod_{i=1}^{d-1}\left( X-a_iY\right) \text{ and } (X-pY)X \prod_{j=1}^{d-2} \left( X- b_jY\right) 
 \end{equation}
 are  not isomorphic.
 \end{proposition}
 
 \begin{proof} The proof is based on classical properties of the cross--ratio of four  points on $\PP^1 (\C)= \C \cup \{ \infty \}$. Recall that if $(x_1,x_2,x_3,x_4)$ is a  quadruple of four distinct complex   numbers, the associated   {\em  cross--ratio } is the complex number $[x_1,x_2,x_3,x_4]$ defined by
  $$
 [x_1,x_2,x_3,x_4]:=  \frac{x_3-x_1}{x_3-x_2} \, \Bigg/ \, \frac{x_4-x_1}{x_4-x_2}\cdotp
  $$
  This definition is naturally extended to $\PP^1 (\C)$  when exactly one of the elements $x_1$, $x_2$, $x_3$ and $x_4$ is equal to $\infty$.  The cross--ratio
  is invariant by any homography  of $\PP^1 (\C)$. In other words, for any homography  $\mathfrak h$, for any quadruple $(x_1,x_2,x_3,x_4)$ of distinct points of
  $\PP^1 (\C)$,  one has the equality
\begin{equation}\label{1045}
  [x_1,x_2,x_3,x_4] = [\mathfrak h (x_1), \mathfrak h (x_2), \mathfrak h (x_3), \mathfrak h (x_4)].
\end{equation}
  Let  $a$ be a nonzero integer. The canonical decomposition of $|a|$ into prime factors
  $$
  |a|=\prod_p p^{v_p(a)}
  $$
defines, for each prime number $p$, the   {\em  $p$--adic valuation  } $v_p(a)\in\Z$ of $a$. 
Let  $t=a/b\ne 0$ be a rational number, written in its irreducible form. The   {\em  $p$--adic valuation  } of $t$, is the non negative integer 
$$
v_p (t): =
\begin{cases}
v_p (a) & \text{ if } p\nmid b,\\
-v_p (b) &\text{ if  } p\nmid a.
\end{cases}
$$
We now begin the proof of  Proposition \ref{1026}. This proof is  by contradiction. Let  $F_1(X,Y)$ and $F_2(X,Y)$ respectively be the two binary
forms introduced in \eqref{defFG}. 
Suppose that there is  $\gamma \in {\Gl}(2, \Q)$,  written as in \eqref{defgamma}, such that 
$$
F_1 
=
F_2 \circ \gamma.
$$
Then the homography 
  $\mathfrak h$ associated with $\gamma$ has the shape 
  $$
  z\mapsto \mathfrak h (z) =\frac {az+b}{cz+d}\cdotp
  $$
This homography induces a bijective map between  the sets of zeroes of the polynomials   $
f_1(X):=F_1  (X,1)
$
and 
$
f_2 (X):=F_2(X,1).
$
These sets of zeroes are  $\calZ (f_1):= \{ 0, \, a_1,\dots , a_{d-1}\}$ and  $\calZ (f_2)=\{ 0, \, b_1, \dots, b_{d-2}, \, p\}$ considered as subsets of $\PP^1 (\C)$. Consider, for $j=1,\, 2$,    the subsets of $\Q\setminus\{0\}$ defined by    
 \begin{equation}\label{defBir(f)}
\Bir (f_j):=
\left\{ [x_1,x_2,x_3,x_4]: x_i \in \calZ (f_j), x_i \text{ distinct}
\right\}. 
\end{equation}
The equality \eqref{1045} implies the equality of the two sets
$$\Bir( f_1) = \Bir (f_2),$$
and also of the two sets
$$
\{v_p (y): y \in \Bir (f_1)\} =\{ v_p (z) : z \in \Bir (f_2)\}.
$$
As a consequence of the inequalities  \eqref{firstinequality}, we have  $\{v_p (y): y \in \Bir (f_1)\}= \{ 0\}$. However we also have $1\in \{ v_p (z) : z \in \Bir (f_2)\}$ by considering the cross ratio
$ [ 0,b_1,p,b_2]$ and the inequalities \eqref{secondinequality}.
So we reach a contradiction: the element $\gamma$ does not exist and the binary forms  $F_1$ and $F_2$ are not isomorphic. \end{proof}

 \subsection{Triviality of the group  ${\Aut} ( L_{d, p},  \Q)$} \label{6.2} In order to determine the value  of the coefficient $W$ appearing in  Proposition \ref{Stewart},
 we prove the following.
 
\begin{proposition}\label{1139} Let  $d\geq 5$ be an integer. For every prime  $p\geq d$, the automorphism group of the binary form 
 $L_{d,p} $   is  $\{ {\Id}\}$ if  $d$ is odd, and  $\{ \pm {\Id} \}$ if $d$ is even.
 In particular, the set $\calL_d$ fulfils the conditions  {\em  C1}  or  {\em  C2}  of Corollary  \ref{11/21/4},  according to the parity of $d$.

 \end{proposition}
 
  \subsubsection{Two preliminary results}
The proof of the following lemma  is based on the analytic properties of the homography on each of its intervals of definition.
 
\begin{lemma} \label{tentative} Let    $\mathfrak h$  be a homography belonging to ${\PGl}(2 ,\R)$,
$M>0$ be a real number, $t\geq 1$ be an integer, $x_1,\dots,x_t$ be $t$ real numbers    satisfying $0< x_1< \cdots <x_t <M$,
$y_1,\dots,y_t$ be $t$ real numbers satisfying   $0 <y_1< \cdots < y_t <M$.
Assume
 $$
 \begin{cases}
 \mathfrak h \left(\left\{ x_i: 1\leq i \leq t\right\} \right) = \left\{ y_j : 1 \leq j \leq t\right\},\\
 \mathfrak h (0) =0 \text{ and } \mathfrak h (M) =M.
 \end{cases}
 $$
 Then, for every  $1 \leq i\leq t$, one has the equality $\mathfrak h (x_i) =y_i$. 
 \end{lemma}

 \begin{proof} We split the proof in several cases depending on the nature of the homography   $\mathfrak h$.
 \vskip .3cm
 \begin{itemize}
 \item If  $\mathfrak h (\infty) = \infty$, the restriction of  $\mathfrak h$ to the real affine line has the shape
 $\mathfrak h (x) = ax+b$, where  $a\ne 0$ and  $b$ are real numbers. The conditions  $\mathfrak h (0)=0$ and $\mathfrak h (M)=M$ imply
 $a=1$ and $b=0$. Hence the result since   $\mathfrak h$ is the identity.
 \vskip .3cm
\item If  $\mathfrak h (\infty) \ne \infty$, $\mathfrak h$  has a unique expansion as
 \begin{equation}\label{h(x)=a+}
 \mathfrak h(x)= a + \frac b{x-c},
\end{equation}
where  $a,$ $b $ and $c$ are real numbers such that     $c\not\in\{ 0, \, x_1,  \cdots, \, x_t, M\}$ and $b\ne 0$. 
We now consider the respective values of   $b$ and $c$. 
\vskip .3cm
 \begin{itemize}
 \item  If  $b>0 $  the function  $x\mapsto \mathfrak h (x)$ is decreasing on the two intervals 
  $(-\infty, c )$ and $(c , + \infty)$. We consider the value of  $c$.
 \begin{itemize} 
 \item If  $c <x_t \ (<M)$,  we have   the inequality  $ \mathfrak h (x_t) > \mathfrak  h(M) =M,$  since $\mathfrak h$ is decreasing. 
 This contradicts the hypothesis     $\mathfrak h (x_t) <M$.
  \item If  $c >x_t  \ (>0)$,  we have
  $
  0=\mathfrak h (0) > \mathfrak h(x_t) $. This contradicts the hypothesis  $\mathfrak h (x_t)  >0$. We conclude that   $\mathfrak h$ is not of the form  \eqref{h(x)=a+} with 
  $b >0$.
  
 \end{itemize}
   
 \vskip .3cm
\item If  $b <0$, the function $x\mapsto \mathfrak h (x)$ is increasing on both intervals $(\infty, c)$ and $(c, +\infty)$. We now consider the value of $c$.
\begin{itemize}
\item If $c\not\in[0,M]$, the function $x\mapsto \mathfrak h (x)$ is increasing on  $(0, M)$, so   we have $\mathfrak h (x_i)=y_i$ for $1 \leq i\leq t$.

\item If  $0 < c <M$, the hyperbola $\{(x, \mathfrak h (x))\in \R^2: x \in \R, \, x\ne c\}$ has two asymptotes: one with abscissa $c$ and 
the other one with ordinate $a$.
Elementary considerations on this hyperbola lead to the  inequalities
$$
 \mathfrak h (0) >a > \mathfrak h (M).
 $$
 This contradicts the  hypothesis   $\mathfrak h (0)=0$ and  $\mathfrak h (M)=M$. In conclusion 
  $\mathfrak h$ is not of the form  \eqref{h(x)=a+} with $b<0$ and  $0< c <M$.
 \end{itemize}
 \end{itemize}
 \end{itemize}
 \end{proof}

 We will require the following  variant of Lemma \ref{tentative}. 
  
\begin{lemma}\label{tentative2} Let  
 $\mathfrak h$  be a homography belonging to ${\PGl}(2 ,\R)$,
$M>0$ be a real number,
$t\geq 1$ be an integer,
$x_1,\dots,x_t$ be $t$ real numbers satisfying $0< x_1< \cdots <x_t <M$,
$y_1,\dots,y_t$ be $t$ real numbers satisfying   $0 <y_1< \cdots < y_t <M$.
Assume
 $$
 \begin{cases}
 \mathfrak h \left(\left\{ x_i: 1\leq i \leq t\right\} \right) = \left\{ y_j : 1 \leq j \leq t\right\},\\
 \mathfrak h (0) =M \text{ and } \mathfrak h (M) =0.
 \end{cases}
 $$

Then for every  $1 \leq i\leq t$, one has the equality  $\mathfrak h (x_i) =y_{t+1-i}$. 
 \end{lemma}
 
 \begin{proof} Introduce the homography 
  $\mathfrak g =\mathfrak s \circ \mathfrak h,$ where  $\mathfrak s $ is the symmetry  $\mathfrak s (x) =M-x$.
  The homography  $\mathfrak g$ fulfils the hypotheses of Lemma  \ref{tentative} 
  provided that we replace the points  $y_i$ ($1\leq i \leq t)$  by the points  $y'_i := M-y_{t+1-i}$.
We deduce that for all $i$ one has the equality  $\mathfrak g (x_i) =y'_i$, which gives  $\mathfrak h (x_i) =y_{t+1-i}$.
 \end{proof}

 \subsubsection{Proof of  Proposition \ref{1139}}
 
 \begin{proof}  Consider the polynomial 
 $$
 f(X) = L_{d,p} (X,1)$$
 and  its set of zeroes $\calZ(f) =\{0, \, 1,\, 2,\cdots,\,  d-2, \, p\}$.
In order to prove that the group of automorphisms of  $L_{d,p}$ is trivial
it suffices to prove that the unique homography  $\mathfrak h \in {\PGl }(2, \Q)$, such that  
\begin{equation}\label{h(Z)=Z}
\mathfrak h\left( \calZ (f)\right)= \calZ(f),
\end{equation}
is the identity as soon as the prime $p$ satisfies  $p\geq d$. 

As in the proof of  Proposition \ref{1026}, we will play with  the  $p$--adic valuation  of the elements  in $\Bir (f)$, defined in  \eqref{defBir(f)}.
We first notice that  for $x$ and  $y$ two distinct integers in $\{1, \, 2, \dots, d-2\}$, the elements  
$$ 
 \alpha:= [0, x, p ,y ], \, [p, x, 0 , y], [x,0,y, p] \text{ and } [x, p, y, 0],
$$ 
belong to $\Bir (f)$ and satisfy $v_p (\alpha) =1$. These are the only elements in $\Bir (f)$ which 
satisfy $v_p (\alpha) =1$.  In particular, if four distinct elements  $x,y,z,t$ in $\calZ(f)$ satisfy $v_p([x,y,z,t])=1$, then  $\{0,p\}\subset  \{x,y,z,t\}$.

 By  \eqref{1045}, we have 
 the following equality  
$$ 
 v_p\left( [\mathfrak  h(x),\mathfrak h (0), \mathfrak h( y) , \mathfrak h(p) ]\right) =  1,
$$ 
where $x$ and $y$ are integers as above.  Since $d\ge 5$, there exists an  integer   $x$  in $\{1, \, 2, \dots, d-2\}$ such that $\mathfrak  h(x)
\not\in \{0,p\}$.  We claim that there is another integer  $y\not=x$  in $\{1, \, 2, \dots, d-2\}$ with the same property, namely such that $\mathfrak  h(y)
\not\in \{0,p\}$. This is plain for $d\ge 6$; for $d=5$, the only case where this would not be true is when $\{1,2,3\}=\{x,y,z\}$ with $\{\mathfrak  h(y),\mathfrak  h(z)\}= \{0,p\}$, but this case is not possible since it would not be compatible with our requirement that 
$$
\{0,p\}\subset\left\{ \mathfrak  h(x),\mathfrak h (0), \mathfrak h( y) , \mathfrak h(p)\right\}.
$$
This proves our claim that there are two distinct integers   $x$ and  $y$  in the set $\{1, \, 2, \dots, d-2\}$ such that $\{ \mathfrak  h(x),\mathfrak  h(y)\}\cap \{0,p\}=\emptyset$. 
Therefore 
\begin{equation*} 
\{\mathfrak h (0),\mathfrak h(p)\}= \{0,p\}.
\end{equation*}
We consider two cases. 
   \begin{enumerate}[label=\upshape(\roman*), leftmargin=*, widest=iii]
\item 
Assume
\begin{equation*} 
\mathfrak h (0) =0 \text{ and } \mathfrak h (p) =p.
\end{equation*}
  Since  $\mathfrak h$ induces by restriction a bijective map of  $\calZ (f)$ onto itself, we may apply  Lemma 
  \ref{tentative}. We deduce that  $\mathfrak h (t) =t$ for $0\leq t \leq d-2$ and  $\mathfrak h (p)=p$.  Since a homography is determined by its restriction to a set with three elements, we deduce that 
  $\mathfrak h = {\Id}$. 
\item  
If
   \begin{equation}\label{cond1.0}\mathfrak h (0) =p \text{ and } \mathfrak h (p) =0,
\end{equation}
we apply  Lemma  \ref{tentative2} to deduce that $\mathfrak h (i) = d-1-i$, for  $1\leq i \leq d-2$. The unique homography $\mathfrak h$ satisfying this property is the symmetry 
defined by $\mathfrak h : z \mapsto d-1-z$. But such a formula is not compatible with the fact that $\mathfrak h (0) =p$. So there is no homography $\mathfrak h$ 
satisfying \eqref{h(Z)=Z}  and  \eqref{cond1.0}.

 \end{enumerate}
  
 We conclude that the set of $\mathfrak h$
 satisfying \eqref{h(Z)=Z} is reduced to the identity.
 The proof of Proposition \ref{1139} is complete. 
 \end{proof}

 \subsection{The family $\calL$ is regular (continued)}\label{SS:Lregular}

 We now investigate the condition {\rm (v)} of Definition \ref{agreable}. We will prove

\begin{proposition}\label{1320}
   For every  $d\geq 5$,  
     for every  $p$ with $p\geq d-1$,  and for all  $(x,y)\in \Z^2$ such that $L_{d,p}(x,y) \ne 0$.  the following inequality holds 
  \begin{equation}\label{majoration}
 \max\{ \vert x\vert ,\vert y \vert \}\leq   
 9 \cdot  \vert\,L_{d,p}(x,y)\,\vert^{\frac 1{d-1}}.
 \end{equation}

\end{proposition}

The inequality   \eqref{majoration} is equivalent to the lower bound 
 \begin{equation}\label{minoration}
 \vert L_{d,p} (x,y)\vert \geq  
\left(\,
\frac 1 9 \cdot \max\{\vert x\vert, \vert y \vert\} \, 
\right)^{d-1},
\end{equation}
under the hypotheses of Proposition \ref{1320}. We will rather  work with \eqref{minoration}.
 
 The proof of \eqref{minoration} depends on the relative sizes of $\vert x \vert$ and $\vert y \vert$. However, if we suppose that $xy\le 0$ and $ L_{d,p}(x,y)\not=0$, it is straightforward 
 to obtain the lower bound 
 $$
 \vert L_{d,p}(x,y)  \vert \geq \left( \max\{\vert x \vert, \vert y \vert\}\right)^{d-1}.
 $$
 Hence we may assume that $x$ and $y$  are not zero and have the same sign. Besides, since $\vert L_{d,p}(-x,-y)  \vert=\vert L_{d,p}(x,y)  \vert$, we will assume that both $x$ and $y$ are positive.

The basic equality  is the following one  
 \begin{equation}\label{vertLvert=}
\vert L_{d,p}(x,y) \vert =   x \cdot \vert x-y \vert \cdot \vert x-2y \vert \cdots \vert x-(d-2) y \vert \cdot \vert x -p y\vert.
\end{equation} 
We split the argument according to the relative sizes of $ x $ and $ y $.
 
\subsubsection{Assume $1\le  x  \leq  y $}  
Let $x$ and $y$ be positive integers such that $L_{d,p}(x,y) \ne 0$ with $y\ge x$. Hence $y\ge x+1$. We  
deduce from \eqref{vertLvert=}
$$
\begin{aligned}
\vert L_{d,p}(x,y) \vert &=   x \cdot (y-x) \cdot (2y-x)\cdots ((d-2) y -x)\cdot  (p y-x)
\\
&> x \cdot (y-x)\cdot   y   \cdot (2y) \cdots ((d-3)  y )\cdot ((p-1)y)
\\
&= x \cdot (y-x)\cdot (d-3)\, !\cdot (p-1)\,   y^{d-2}.
\end{aligned}
$$
If $y\ge 2x$ we have $x(y-x)\ge y-x\ge  y/2$, while for $x<y\le 2x$ we have $x(y-x)\ge  x\ge  y/2$. Hence
$$
\vert L_{d,p}(x,y) \vert >  \frac 1 2 (d-3)\, !(p-1)\left( \max\{ \vert x \vert, \vert y \vert \} \right)^{d-1}.
$$

So we proved

\begin{proposition}\label{Case1}
For every $d\geq 3$, for every $p\geq d-1$, for  every integers $x$ and $y$ such that $L_{d,p} (x,y) \ne 0$ and 
$\vert x \vert \leq \vert y \vert$, one has the inequality
$$
\vert L_{d,p}(x,y) \vert \geq   \max\{ \vert x \vert, \vert y \vert\}^{d-1}.
$$
\end{proposition}

\subsubsection{Assume $   (d-2) y  \leq  x  $} 
Let $x$ and $y$ be positive integers such that $L_{d,p}(x,y) \ne 0$ with $x\ge    (d-2) y$, hence $x\ge    (d-2) y+1 $. 
We  
deduce from  \eqref{vertLvert=}
$$
\vert L_{d,p}(x,y) \vert  =x \cdot (x-y )\cdot (x-2y )\cdots ( x-(d-2) y ) \cdot \vert x -p y\vert.
$$

\noindent
$\bullet$ If $y=1$, since  $x\ge d-1$, we have
$$
x-n = x\left(1-\frac n x\right)\ge  x\left(1-\frac n {d-1}\right)= x\left(\frac {d-n-1} {d-1}\right)
$$
for $0\le n\le d-2$; using the trivial lower bound  $\vert x -p \vert\ge 1$ together with Stirling's formula \eqref{Equation:Stirling}, we deduce 
$$
 |L_{d,p}(x,1)|\ge x \cdot (x-1 )\cdot (x-2 )\cdots ( x-(d-2)  )\ge \frac{(d-1)!}{(d-1)^{d-1}} \, x^{d-1}\ge \frac {x^{d-1}}{\rme^{d-1}}\cdotp
 $$

\noindent
$\bullet$ We assume now $y\ge 2$. As a consequence of the hypothesis $y\leq x/(d-2)$, we have the inequality
 $$
 x \cdot (x-y )\cdot (x-2y )\cdots ( x-(d-3) y )\ge \frac{(d-2)!}{(d-2)^{d-2}} \, x^{d-2}.
 $$

\noindent
$\diamond$  If $x>py$, then 
 $$
 x-(d-2) y \ge   x \left(1-\frac{d-2} p\right)\ge x \left(1-\frac{d-2} {d-1}\right) = \frac x {d-1}
 $$
 and the trivial lower bound $  x -p y \ge 1$ suffices to deduce
 $$
  L_{d,p}(x,y) \ge 
 \frac{(d-2)!}{(d-1)(d-2)^{d-2}} \, x^{d-1}.
$$

\noindent
$\diamond$ If $py>x$, then from $x-(d-2) y \ge 1$ and  $p y-x\ge 1$ we deduce
$$
 ( x-(d-2) y ) \cdot (py-x)\ge  ( x-(d-2) y )+ (p y-x)-1\ge y(p-d+2)-1.
 $$
 If $p=d-1$ we use the assumption $y\ge 2$ which yields
 $$
 y(p-d+2)-1= y-1\ge \frac y 2> \frac x {2p}= \frac x {2(d-1)},
 $$
 while for $p\ge d$ we use the lower bounds
  $$
 y(p-d+2)-1\ge  y(p-d+1) \ge py \left(1-\frac {d-1} p\right) > x \left(1-\frac {d-1} d \right) = \frac x d\cdotp
 $$
 
 Therefore, for $(d-2) y  \leq  x  $ and $y\ge 2$, we have
 $$
  | L_{d,p}(x,y)|  \ge   \frac{(d-2)!}{2(d-1)(d-2)^{d-2}} \, x^{d-1}\ge 
    \frac {x^{d-1}} {2d\rme^{d-2}} \cdotp
 $$
  
 We deduce 
 
\begin{proposition}\label{lemme6.5*}  For  $d\geq 3$,   $p$ prime $\geq d-1$ and  $(x,y) \in \Z^2$ such that 
$\vert x \vert \geq   (d-2)   \vert y \vert$ and 
$L_{d,p} (x,y) \ne 0$
we have 
$$
\left\vert L_{d,p} (x,y)\right\vert \geq  \frac 1 {d\rme^d}   \max\{ \vert x \vert, \vert y \vert\}^{d-1}.
$$
\end{proposition}

\subsubsection{Assume $  (n-1) y  \leq  x  \le ny $ for some $n$ with $2\le n\le d-2$}  
We  
deduce from \eqref{vertLvert=}
$$
\vert L_{d,p}(x,y) \vert  =x \cdot (x-y ) \cdots
 (x-(n-1)y )\cdot (ny-x)\cdots
 ( (d-2) y-x ) \cdot  (p y-x).
$$
We have 
$$
x \cdot (x-y ) \cdots
 (x-(n-2)y )
 \ge (n-1)!y^{n-1}
 $$
 and
 $$
 \begin{aligned}
 ((n+1)y-x)\cdots
 ( (d-2) y-x ) \cdot  (p y-x)&\ge (d-n-2)!(p-n)y^{d-n-1}
 \\
 &\ge (d-n-1)! y^{d-n-1}.
 \end{aligned}
 $$
For the product of the two terms in the middle, if  $y=1$ we use the trivial  lower bound $ (x-(n-1)y )  (ny-x)\ge 1$ which yields  
 $$
 \vert L_{d,p}(x,y) \vert \ge (n-1)! (d-n-1)!  y^{d-2} \ge \frac{(n-1)! (d-n-1)!} {n^{d-2}} x^{d-2},
 $$
 while for $y\ge 2$ we use
 $$
 (x-(n-1)y )  (ny-x)\ge (x-(n-1)y ) + (ny-x)-1=y-1\ge \frac y 2,
 $$
 which yields 
$$
 \vert L_{d,p}(x,y) \vert \ge   \frac 1 2 (n-1)! (d-n-1)!  y^{d-1}\ge 
  \frac{ (n-1)! (d-n-1)!} {2n^{d-1}} x^{d-1}.
 $$
We now use Lemma \ref{Lemme:minoration}:
$$
\frac{(n-1)!(d-n-1)!}{n^{d-1}}=\frac{ n!(d-n)!}{n^d(d-n)}\ge \rme^{-(1+\rme^{-1})d}\frac 1 {d-n}, 
$$
from which we deduce 
$$
 \vert L_{d,p}(x,y) \vert \ge    
\rme^{-(1+\rme^{-1})d}\frac 1 {2(d-n)}  x^{d-1}.
 $$
This proves the following result:

\begin{proposition}\label{lemme6.6*}
 For  $d\geq 3$, $2\le n\le d-2$,  $p$ prime $\geq d-1$  and $x$ and  $y$ such that
$(n-1)|y|\le \vert x \vert \le  n \vert y \vert$ and 
$L_{d,p} (x,y) \ne 0$,
we have  
$$\left\vert L_{d,p} (x,y)\right\vert \geq \frac 1 {2(d-2)} \cdot \frac{ \max\{ \vert x \vert, \vert y \vert\}^{d-1}}{\rme^{(1+\rme^{-1})d}}\cdotp
$$
\end{proposition}
 
For $d\ge 5$, we have 
$$
2(d-2)\cdot \rme^{(1+\rme^{-1})d}<9^{d-1}.
$$

We may now gather  Propositions \ref{Case1}, \ref{lemme6.5*}  and  \ref{lemme6.6*}
to deduce \eqref{minoration}, which completes the proof of Proposition \ref{1320}.

\subsection{Estimating the number  of images by   $\calL$ of $(x,y)$ with $\max \{ \vert x \vert, \vert y \vert\} \ge 10$} 

 Gathering  Propositions \ref{1026} and   \ref{1320}, we proved that the family $\calL$ is $(10, 1, 1, 5, 9)$--regular. Furthermore, according to the parity of $d$,  the set $\calL_d$
 satisfies the conditions  {\rm  C1} or  {\rm  C2} of Corollary \ref{11/21/4}, by Proposition \ref{1139}. As a consequence of Corollary \ref{11/21/4} we have the following 
 
 \begin{proposition}\label{1498} For any $d \geq 5$, for every $\varepsilon >0$,  one has the equality 
 $$
 \calR_{\geq d} \left( \calL, B, 10\right)= \frac 1{(2,d)}\left( \sum_{d\leq p < 2d} A_{L_{d,p}}\right) B^{2/d} +O _{d, \varepsilon}\left( B^{\vartheta_d + \varepsilon}\right) + O_d\left( 
 B^{2/(d+1)}\right). 
 $$
 \end{proposition}

\subsection{Estimating the number  of images by   $\calL$ of $(x,y)$ with $\max \{ \vert x \vert, \vert y \vert\} < 10$} 

The difference
\begin{equation}\label{differenceL}
\mathcal R_{\geq d} (\mathcal L, B,0)-\mathcal R_{\geq d} (\mathcal L, B,10)
\end{equation} is bounded from above by  two times the cardinality
of the set
\begin{multline*}
\mathfrak {Er}_{\geq d} (B) \\ := \{ m\; :\;0 <  m=\vert L_{d',p} (x, y) \vert \leq B, \;   d\le d'\leq p < 2d',\, \max \{ \vert x \vert, \vert y \vert\}  \leq 9\}.
\end{multline*}
There are  $19^2$ pairs $(x,y)$ with $\max \{ \vert x \vert, \vert y \vert\} \le 9$. We first count the number of $m$ in $\mathfrak {Er}_{\geq d} (B)$ of the form   $\vert L_{d',p} (x, 0) \vert$, namely with $y=0$. For $x=\pm 1$ and $y=0$ we have $m=1$; for $2\le |x|\le 9$ and $y=0$, we have  $2^{d'}\le B$, hence there are at most $O_d(\log B)$ such values of $m$. 

We count now the number of  $m$ in $\mathfrak {Er}_{\geq d} (B)$ of the form $\vert L_{d',p} (x, y) \vert$ with $|y|\ge 1$. We have $|x-ny|\ge n-|x|\ge n-9\ge 2$ for $n\ge 11$, hence 
$$
B\ge m\ge \prod_{11\le n\le d'-2} (n-9)\ge 2^{d'-12}, 
$$
and therefore  $d'\le O( \log B)$.  It follows that  the number of pairs $(d',p)$ as above is bounded by  $O_d(\log^2 B)$. So we proved  
 $$
 \sharp \, \mathfrak{Er}_{\geq d} (B) =O_d(\log^2 B).
 $$
Combining this bound with \eqref{differenceL} and with  Proposition \ref{1498}, we obtain the equality \eqref{418} of  Theorem \ref{CaseofL}.


\subsection{Some results on $A_F$ for $F\in \calL$}\label{areaonL} The area of the fundamental domain associated to $L_{d, p}$ is,
by the definition \eqref{defAF}, equal to 
$$
A_{L_{d,p}} = \iint_{\calD(L_{d,p})}{\rmd} x\, {\rmd}  y,
$$
with
$$
\calD (L_{d,p}):= \bigl\{ (x,y) \in \R^2 : \left\vert x (x-y)(x-2y) \cdots (x-(d-2)y) (x-py)\right\vert \leq 1\bigr\}.
$$
 By the change of variables  $u=x$ and  $v=y/x$, we obtain
 $$
A_{L_{d,p} }= \iint_{\calD^*(L_{d,p})}\vert u\vert \, {\rmd} u\, {\rmd}  v,
$$
 with
 $$
 \calD^* (L_{d,p}):= \bigl\{ (u,v) \in \R^2 : \vert u \vert^d\cdot \left\vert  (1-v)(1-2v) \cdots (1-(d-2) v) (1-pv)   \right\vert \leq 1
 \bigr\}.
 $$
 Some elementary calculations transform $A_{L_{d,p}}$ into a single integral.
 
\begin{lemma} 
 For  $d \geq 5$ and  $p\geq d-1$ the following equalities hold
 $$
 A_{L_{d,p}}= \int_{-\infty}^\infty \frac{ {\rmd}  v}{\left(\, \vert 1-v\vert \cdot \vert 1-2 v\vert \cdots \vert 1-(d-2)v\vert \cdot \vert 1-p v\vert
 \,\right)^{2/d} }
 $$
 and 
  $$
 A_{L_{d,p}}= \int_{-\infty}^\infty \frac{ {\rmd}  t}{\left(\, \vert t \vert \cdot \vert t-1\vert \cdot \vert t-2  \vert \cdots \vert t-(d-2) \vert \cdot \vert t-p  \vert
 \,\right)^{2/d} }\cdotp
 $$
 \end{lemma}
We will only work with the second expression of $A_{L_{d,p}}$.  So we introduce the function
 $$
 \lambda_{d,p} (t) :=  t(t-1)\cdots (t-(d-2))(t-p),
 $$
 which is the product of $d$ linear factors in $t$.
We split the interval of integration into $d$  intervals of length $1$ around the singularities $0$,..., $d-2$ and $p$ and three remaining intervals 
to write the equality:
\begin{multline}\label{decomposition}
A_{L_{d,p}} :=
\\
\left( \int_{-\infty}^{-1/2}+\int_{-1/2}^{1/2}+ \cdots+\int_{d-(5/2)}^{d-(3/2)} +\int_{d-(3/2)}^{p-(1/2)}+ \int_{p-(1/2)}^{p+(1/2)} +\int_{p+(1/2)}^{\infty}\right) \frac{{\rmd} t}{\vert \lambda_{d,p} (t)\vert^{2/d}}\cdotp
\end{multline}
We will give an upper bound and a lower bound for each of these positive  integrals in order to prove

\begin{proposition}\label{goal}  Uniformly for $d \to\infty$ and $d\leq p < 2d$ one has 
$$ 
\frac {\rme^2-o(1)}d \le A_{L_{d,p}} \le \frac {5\,\rme^2+2\rme+o(1)}d\cdotp
$$ 
\end{proposition}

 The last part of Theorem \ref{CaseofL} is obtained from this proposition after a summation over $d\leq p <2d$ and an application of the  Prime Number Theorem. 

\subsubsection{An auxiliary lemma}

\begin{lemma}\label{Lemma:factorielles}
For $d\to\infty$, we have 
$$
\left(1\cdot 3\cdot 5\cdots(2d-3)\right)^{1/d}=(2\rme^{-1}+o(1))d.
$$
\end{lemma}

\begin{proof}
We write
$$
 1\cdot 3\cdot 5\cdots(2d-3) = \frac {(2d-3)!}{2^{d-2}(d-2)!}=
\frac {(2d)!}{(2d-1)2^dd!} 
$$
and we use Stirling's formula \eqref{Equation:Stirling} which gives
$$ 
 \Bigl( \frac {2d}{\rm e}\Bigr)^d \frac{\sqrt 2}{(2d-1) \cdot {\rm e}^{1/12d}}\leq 1\cdot 3 \cdot 5\cdots (2d-3) \leq \Bigl( \frac {2d}{\rm e}\Bigr)^d \frac{\sqrt 2\cdot {\rm e}^{1/24d}}{2d-1}\cdotp
$$ 
\end{proof}

\subsubsection{Study of $\int_{-\infty}^{-1/2}$ and of $\int_{p+1/2}^\infty$}

\begin{lemma}\label{chr1}
For $d\to\infty $ and $p\ge d$, one has
$$
0\le   \int_{-\infty}^{-1/2} \frac{{\rmd}t}{\vert \lambda_{d,p} (t)\vert^{2/d}} \leq   \frac {\rme+o(1)}{d}\cdotp
$$
\end{lemma}

\begin{proof}
Using H\" older inequality and Lemma \ref{Lemma:factorielles}, we obtain
\begin{align*}
&\int_{-\infty}^{-1/2} \frac{{\rmd}t}{\vert \lambda_{d,p} (t)\vert^{2/d}}\leq \\
&\left( \int_{-\infty}^{-1/2}
\frac{{\rmd}t}{\vert t \vert^2}\right)^{1/d} 
\left( \int_{-\infty}^{-1/2}
\frac{{\rmd}t}{\vert t -1\vert^2}\right)^{1/d} \cdots 
\left( \int_{-\infty}^{-1/2}
\frac{{\rmd}t}{\vert t -(d-2) \vert^2}\right)^{1/d} \left( \int_{-\infty}^{-1/2}
\frac{{\rmd}t}{\vert t -p \vert^2}\right)^{1/d}
\\
& \leq  \left( \frac 21\cdot \frac 23\cdot \frac 25\cdots \frac 2{2d-3}\cdot \frac 2{2p+1}\right)^{1/d} 
\leq \left(\frac{ 2^d }{1\cdot 3 \cdot 5\cdots  (2d-3) \cdot(2p+1)}  \right)^{1/d}
\\
&\leq   \frac {\rme+o(1)}{d}\cdotp
\end{align*}
\end{proof}
Similarly, one proves 

\begin{lemma}\label{chr2}
For $d\to\infty$ and $p\ge d$, one has
$$
0\le \int_{p+(1/2)}^{\infty} \frac{{\rmd}t}{\vert \lambda_{d,p} (t)\vert^{2/d}}
\le
  \frac {\rme+o(1)}{d}\cdotp 
$$
\end{lemma}

\begin{proof}
For $t>p+\frac 12$, we have 
$$
|\lambda_{d,p} (t) |= \lambda_{d,p} (t)=  t ( t -1)\cdots ( t -(d-2))( t -p).
$$
Using H\" older inequality and Lemma \ref{Lemma:factorielles}, we obtain 
\begin{align*}
&
 \int_{p+(1/2)}^{\infty} \frac{{\rmd}t}{\vert \lambda_{d,p} (t)\vert^{2/d}}\leq \\
&\left(  \int_{p+(1/2)}^{\infty}
\frac{{\rmd}t}{  t^2}\right)^{1/d} 
\left(  \int_{p+(1/2)}^{\infty}
\frac{{\rmd}t}{ ( t -1)^2}\right)^{1/d} \cdots 
\left(  \int_{p+(1/2)}^{\infty}
\frac{{\rmd}t}{( t -(d-2))^2}\right)^{1/d} \left(  \int_{p+(1/2)}^{\infty}
\frac{{\rmd}t}{( t -p)^2}\right)^{1/d}
\\
& \leq  \left( \frac 2{2p+1}\cdot \frac 2{2p-1}\cdot \frac 2{2p-3}\cdots \frac 2{2p-2d+5}\cdot \frac 2 1\right)^{1/d} 
\\
&\leq \left(\frac{ 2^d }{1\cdot 3 \cdot 5\cdots  (2d-3)  }  \right)^{1/d}
\\
&\leq   \frac {\rme+o(1)}{d}\cdotp
\end{align*}
\end{proof}

\subsubsection{Study of $\int_{d-3/2}^{p-1/2}$}\label{6.3.2} 

\begin{lemma}\label{chr3}
For $d\to\infty$ and $d\leq p < 2d$, one has
$$
0 \le \int_{d-3/2}^{p-1/2} \frac{{\rmd}t}{\vert \lambda_{d,p} (t)\vert^{2/d}} \le \frac {\rme^2 +o(1)}d\cdotp
$$
\end{lemma}

\begin{proof}
For $t$ in the interval $(d-(3/2),p-(1/2))$,  we have $p-t> 1/2$,
$$
|\lambda_{d,p} (t) |=  t ( t -1)\cdots ( t -(d-2))( p-t)
$$ 
and,  for $0\le n\le d-2$,  
$$
t-n>\frac{2d-2n-3} 2,
$$
hence
$$
\vert\lambda_{d,p} (t)\vert\ge\frac{(2d-3)\cdot(2d-5)\cdots 3\cdot 1}{2^d}
$$
and therefore 
$$
\vert\lambda_{d,p} (t)\vert^{2/d}\ge(\rme^{-2}+o(1))d^2
$$
by  Lemma \ref{Lemma:factorielles}. 
Since $d\le p< 2d$,  the interval of integration has length at most $d+1$,  and so we deduce
$$
\int_{d-3/2}^{p-1/2} \frac{{\rmd}t}{\vert\lambda_{d,p} (t)\vert^{2/d}}
 \le\frac {\rme^2 +o(1)}d\cdotp
$$
\end{proof}

\subsubsection{Study of $\int_{p-1/2}^{p+1/2} $}

\begin{lemma}\label{chr3*}
For $d\geq 5$ and $d\leq p < 2d$, one has
$$
0\le \int_{p-1/2}^{p+1/2} \frac{{\rmd}t}{\vert \lambda_{d,p} (t)\vert^{2/d}} =O\left( \frac 1 {d^2}\right). 
$$
\end{lemma}

We introduce the polynomial 
$$
{\calM}( t) := t(t-1)\cdots (t-(d-2))
$$
of degree $d-1$. 
It is easy to see that 
$$
\min_{\vert t -p\vert \leq 1/2} \vert {\calM} (t)\vert =  \vert {\calM} (p-(1/2))\vert \geq {\calM} (d-(3/2)) =  \frac 12 \cdot \frac 32\cdot \frac 52 \cdots \frac{2d-5}2\cdot  \frac{2d-3}2  ,
$$
hence by Lemma  \ref{Lemma:factorielles}, we have
$$
\min_{\vert t -p\vert \leq 1/2} \vert {\calM} (t)\vert ^{2/d}\ge (\rme^{-2}+o(1))d^2.
$$
Since
$$
\int_{p-1/2}^{p+1/2} \frac{{\rmd}t}{\vert t -p\vert^{2/d}}=O(1)
$$
we conclude
$$
\begin{aligned}
\int_{p-1/2}^{p+1/2}\frac{{\rmd}t}{\vert \lambda_{d,p} (t)\vert^{2/d}} 
&
\leq 
\left(\int_{p-1/2}^{p+1/2} \frac{{\rmd}t}{\vert t -p\vert^{2/d}}\right) \cdot \left( \frac 1{\min_{\vert t-p\vert \leq 1/2}\vert {\calM} (t)\vert} \right)^{2/d}
\\
&=O\left( \frac 1 {d^2}\right). 
\end{aligned}
$$ 
  
\subsubsection{Study of the remaining integrals} 
We are now concerned, for $\nu=0, 1 \dots, d-2$, with the integrals
$$
\calI_\nu = \calI_{d,p, \nu}=\int_{\nu-1/2}^{\nu +1/2} \frac{{\rmd} t}{\vert \lambda_{d,p} (t)\vert^{2/d}},
$$
for which we want to find an upper and a lower bound.  We split the product defining $\lambda_{d,p} (t)$ into four pieces
\begin{equation}\label{lambdalambda}    
\lambda_{d,p} (t) =(t-\nu)\cdot (t-p) \cdot  \lambda_\nu^- (t)  \cdot 
 \lambda_{d,\nu} ^+(t),
\end{equation}
with
$$
\lambda_\nu^- (t) := \prod_{0\leq k < \nu} \left( t -k\right) 
\quad \text{and}\quad 
 \lambda_{d,\nu} ^+(t):= \prod_{\nu < k \leq d-2} \left( t-k\right).
$$ 
We have
\begin{equation}\label{first}
  \calI_\nu  \leq 
\left( \int_{\nu-1/2}^{\nu+1/2} \frac{{\rmd}t}{\vert t-\nu\vert^{2/d}} \right)\cdot \left(\min  \vert \lambda_{\nu}^- (t)\vert  \right)^{-2/d}\cdot 
\left(\min  \vert \lambda_{d,\nu}^+ (t)\vert  \right)^{-2/d}\cdot \left( \min \vert t-p\vert \right)^{-2/d},
\end{equation}
and
\begin{equation}\label{second}
  \calI_\nu  \geq 
\left( \int_{\nu-1/2}^{\nu+1/2} \frac{{\rmd}t}{\vert t-\nu\vert^{2/d}} \right)\cdot \left(\max  \vert \lambda_{\nu}^- (t)\vert  \right)^{-2/d}\cdot 
\left(\max  \vert \lambda_{d,\nu}^+ (t)\vert  \right)^{-2/d}\cdot \left( \max \vert t-p\vert \right)^{-2/d}
\end{equation}
where all the maximum and minimum are taken for $\nu-1/2\leq t \leq \nu+1/2$. Direct computations transform \eqref{first} and \eqref{second} into
$$
\begin{aligned}
(1-o(1))
\left(\max  \vert \lambda_{\nu}^- (t)\vert  \right)^{-2/d}\cdot 
&
\left(\max  \vert \lambda_{d,\nu}^+ (t)\vert  \right)^{-2/d} 
\le
\calI_\nu  \le 
\\
&(1+o(1))
 \left(\min  \vert \lambda_{\nu}^- (t)\vert  \right)^{-2/d}\cdot 
\left(\min  \vert \lambda_{d,\nu}^+ (t)\vert  \right)^{-2/d}
\end{aligned}
$$
which is also
\begin{align}\label{third}(1-o(1))
\vert \lambda_\nu^-(\nu +1/2)\vert^{-2/d}\cdot 
&
 \vert \lambda_{d,\nu}^+(\nu-1/2)\vert^{-2/d} 
\le \calI_\nu  \le 
\\
\notag
&(1+o(1))
\vert \lambda_\nu^-(\nu -1/2)\vert^{-2/d}\cdot  \vert \lambda_{d,\nu}^+(\nu+1/2)\vert^{-2/d},
\end{align}
uniformly for $d \to\infty$ and $d\leq p <2d$.

For $1\leq \nu \leq d-2$, we have the equalities
$$
\lambda_\nu^-(\nu +1/2) = \frac {(2\nu +1)(2\nu-1) \cdots 3}{2^\nu} =\frac{(2\nu+1)\,!}{2^{2\nu} \cdot \nu\,!}
=\frac{(2\nu)\,!}{2^{2\nu} \cdot \nu\,!}\cdot \frac 1 {2\nu+1},
$$
$$
\lambda_\nu^-(\nu -1/2) = \frac {(2\nu -1)(2\nu-3) \cdots 1}{2^\nu} =\frac{(2\nu-1)\,!}{2^{2\nu-1} \cdot  (\nu\ -1)\,!}
=\frac{(2\nu)\,!}{2^{2\nu} \cdot \nu\,!},
$$
and for $0\leq \nu \leq d-3$, we have
$$
\vert \,\lambda_{d,\nu}^+(\nu+1/2)\, \vert= \frac{ (2d^*-1)(2d^*-3) \cdots 3\cdot 1}{2^{d^*}}= \frac{(2d^*-1)\,!}{2^{2d^*-1}\cdot (d^*-1)\, !}
=\frac{(2d^*)\,!}{2^{2d^*} \cdot d^*\,!},
$$
$$
\vert\, \lambda_{d,\nu}^+(\nu-1/2)\, \vert =\frac{  (2d^*+1)(2d^*-1)\cdots5\cdot 3}{2^{d^*}}=\frac{(2d^*+1)\, !}{2^{2d^*}\cdot d^*\, !}
=\frac{(2d^*)\,!}{2^{2d^*} \cdot d^*\,!}\cdot (2d^*+1),
$$
with the notation  $d^* =d-2-\nu$. 
Furthermore, since we have empty products in the  decomposition \eqref{lambdalambda} ,  we have
\begin{equation}\label{exception}
\lambda_0^- (1/2) =\lambda_0^-(-1/2) = \lambda_{d, d-2}^+ (d-3/2) =\lambda_{d,d-2}^+ (d-5/2) =1.
\end{equation}

The following lemma shows that the inequalities \eqref{third} are sharp.

\begin{lemma}\label{lemma6.8} Uniformly for $0 \leq \nu \leq d-2$ and $d\rightarrow \infty$ one has
$$
1-o(1)\leq \left( \frac{ \vert \lambda_\nu^-(\nu -1/2)\vert \cdot  \vert \lambda_{d,\nu}^+(\nu+1/2)\vert}
{ \vert \lambda_\nu^-(\nu +1/2)\vert \cdot  \vert \lambda_{d,\nu}^+(\nu-1/2)\vert}
\right)^{-2/d} \leq 1+o(1)
$$
\end{lemma}

 \begin{proof} Obvious consequence of the explicit formulas given above.
\end{proof}

For $0\le \nu\le d-2$, let 
$$
\Lambda = \Lambda(d, \nu)  := \vert \lambda_\nu^-(\nu -1/2)\vert^{-2/d}\cdot  \vert \lambda_{d,\nu}^+(\nu+1/2)\vert^{-2/d} 
$$
As a consequence of the explicit formulas of  $\lambda_\nu^-$ and $\lambda_{d,\nu}^+$, we have the equality 
$$
\log \Lambda  =-\frac 2d\Bigl\{ \log ( (2\nu)\, !)  + \log ((2d^*)\,! )  \\-\log (\nu \,!)-\log (d^*\, !) -2d\log 2  +o(d)\Bigr\},
$$
uniformly for $1\leq  \nu \leq d-3$ and $d  \to\infty$. Using Stirling formula  \eqref{Equation:Stirling}, we deduce
$$
\begin{aligned}
-\frac d 2 \cdot \log \Lambda &=   \nu \log \nu +d^* \log d^* -d
+o(d)
\\
& =  \nu\log \nu +(d-\nu) \log (d-\nu) -d
+o(d),
\end{aligned}
$$
hence
\begin{equation}\label{end3}
\log \Lambda   = - \frac 2d \Bigl( \nu\log \nu +(d-\nu) \log (d-\nu)\Bigl) + 2+o(1),
\end{equation}
uniformly for $1\leq \nu\leq d-3$ and $d \to\infty$. By a direct study of the function $f_d$ defined by
$$
f_d: t\in [1, d-1]\mapsto f_d (t)  = t\log t +(d-t) \log (d-t),
$$
we deduce that, for all $1\leq t \leq d-1$, the function   $f_d$ satisfies the inequality
$$
f_d (d/2) = d\log (d/2) \leq f_d(t) \leq f_d (1)=  f_d (d-1) = (d-1)\log (d-1).
$$
Inserting this bound into \eqref{end3}, we obtain that
\begin{equation}\label{fourth}
 -2 \log d +2 -o(1)\le
\log \Lambda (d,\nu) \le  -2 \log d +2\log 2+2+o(1),
\end{equation}
uniformly for $1\leq \nu \leq d-3$.  Actually this formula also holds for $\Lambda (d,0)$ and $\Lambda (d, d-2)$ thanks to the formulas \eqref{exception}.

Combining \eqref{third}, \eqref{fourth} and Lemma \ref{lemma6.8}, we proved

\begin{lemma} \label{1d2}
Uniformly for  $d \to\infty$, $0\leq \nu \leq d-2$ and $d\leq p < 2d$, one has 
$$
 \frac {{\rm e}^2-o(1)}{d^2}\le  \calI_{d,p, \nu } \le \frac {4\, {\rm e}^2+o(1)}{d^2}\cdotp
$$
\end{lemma}
 
  \subsubsection{End of the proof of Proposition \ref{goal}} 
  We split the end of the proof in two parts.
  \vskip .3cm
 \noindent $\bullet$ For the lower bound, we use positivity to write the inequality
 $$A_{L_{d,p}} \geq \sum_{\nu =0}^{d-2} \calI_\nu \ge (d-1) \cdot \frac {{\rm e}^2-o(1)}{d^2} \ge \frac {{\rm e}^2-o(1)}d, 
 $$
as a consequence of \eqref{decomposition}  and Lemma \ref{1d2}.
 
\vskip .3cm
\noindent $\bullet$ For the upper bound, we respectively apply Lemmas \ref{chr1}, \ref{chr2}, \ref{chr3}, \ref{chr3*} and \ref{1d2} to bound each of these terms in \eqref{decomposition}, and we obtain
$$
A_{L_{d,p}} \le \frac {5\, \rme^2+2\,\rme+o(1)}d\cdotp
$$
The proof of Proposition \ref{goal} is now complete. 
This concludes the proof of Theorem \ref{CaseofL}.

\subsection*{Acknowledgements}
We are thankful to the referee for his valuable comments.


   \end{document}